\documentclass[12pt,english]{article}
\usepackage[round]{natbib}
\usepackage[landscape]{geometry}
\usepackage{color}
\usepackage{tikz}
\usepackage[latin1]{inputenc}
\usepackage[english]{babel}
\usepackage{graphicx}
\usepackage{amsmath}
\usepackage{amssymb}
\usepackage{amsthm}
\usepackage{latexsym}
\usepackage{multirow}
\newtheorem{thm}{Theorem}

\newtheorem{lem}{Lemma}

\title{NW-SE expansions of non-symmetric Cauchy kernels on  near staircases and  growth diagrams
  }
\author{Olga Azenhas and Aram Emami}

\date{}

\begin{document}
\maketitle
\begin{abstract}
Lascoux has given  a triangular version of the Cauchy identity where Schur polynomials are replaced by Demazure characters and  Demazure atoms. He has then used the staircase expansion to recover    expansions for all Ferrers shapes, where the Demazure characters and Demazure atoms are under the action of Demazure operators specified by the cells above the staircase. The characterisation of the  tableau-pairs in these last expansions  is less explicit. 
  We give here a bijective proof for expansions over near staircases, where the  tableau-pairs are made explicit. Our analysis   formulates  Mason's RSK analogue, for semi-skylines augmented fillings, in terms of  growth diagrams.

\end{abstract}

\hspace{0.8cm}\small{{\bf Keywords:} non-symmetric Cauchy kernel, Demazure character, Demazure operator, RSK analogue,

\hspace{0.9cm} semi-skyline augmented filling, growth diagram.


\section{Introduction}
\label{sec:in}
Let $\lambda$ be a Ferrers shape and $\rho=(n,\dots,1)$   the biggest staircase inside of $\lambda$. Fix a cell in the staircase $(n + 1,n,\dots,1)$ which
does not belong to $\lambda$. The diagonal passing through this cell
cuts the skew diagram  $\lambda/\rho$ into a North-West (NW) and a South-East  (SE) parts whose row and column  cell indices are encoded as reduced words for $\sigma(\lambda,NW)$  and ${\sigma(\lambda,SE)}$ in  $\mathfrak{S}_n$, respectively. 
 \cite{lascouxcrystal} has given the following expansion of the Cauchy kernel $F_{\lambda}(x,y)=\prod_{(i,j)\in \lambda}
(1-x_iy_j)^{-1}$  over the  Ferrers shape $\lambda$,

$\displaystyle\label{ext}{\scriptsize  F_{\lambda}:= F_{\lambda}(x,y)=\prod_{(i,j)\in \lambda}
(1-x_iy_j)^{-1}=\sum_{\nu\in \mathbb{N}^{n}}(\pi
_{\sigma(\lambda,NW)}\widehat{\kappa}_{\nu}(x))(\pi
_{\sigma(\lambda,SE)}\kappa_{\omega \nu}(y))},\quad (1)$

\noindent where   $\pi_{\sigma(\lambda,NW)}$  ( $\pi_{\sigma(\lambda,SE)}$) is the Demazure operator 
 indexed by $\sigma(\lambda,NW)$   (${\sigma(\lambda,SE)}$)
acting on the Demazure atom $\widehat\kappa_{ \nu}(x)$ (key polynomial or Demazure character $\kappa_{\omega \nu}(y)$).
He shows that the staircase $\rho$ expansion,
$F_{\rho}=
\sum_{\nu\in \mathbb{N}^{n}}\widehat{\kappa}_{\nu}(x)\kappa_{\omega \nu}(y)$,
allows to recover all the $F_\lambda$ since divided differences or Demazure operators $\pi_i^x$ and $\pi_j^y$  increase the number of poles in the rational function $(1-x_iy_j)^{-1}$:
$\pi_i^x (1-x_i y_j)^{-1}=(1-x_i y_j)^{-1}(1-x_{i+1} y_j)^{-1}$ and similarly for $ \pi_j^y $.
If $\lambda$ consists of the staircase $\rho$ with the sole  cell  $(r+1,e+1)$, $e=n-r$, above $\rho$,  the blue box in
$\begin{array}{ccccccc}
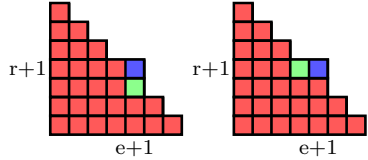
\begin{tikzpicture}[scale=0.25]
\filldraw[color=green!45] (4.07,2.07)rectangle (4.93,2.93);
 \filldraw[color=blue!65] (4.07,3.07)
rectangle (4.93,3.93); \node at (-1.2,3.5){\scriptsize{r+1}}; \node at (4.5,-0.7){\scriptsize{e+1}};
\filldraw[color=red!65] (0.06,0.06) rectangle (5.93,1.93);
\filldraw[color=red!65] (0.06,2.06) rectangle (3.93,3.93);
\filldraw[color=red!65] (0.06,4.06) rectangle (2.93,4.93);
\filldraw[color=red!65] (0.06,5.06) rectangle (1.93,5.93);
 \filldraw[color=red!65] (0.06,6.06) rectangle (0.93,6.93);
  \filldraw[color=red!65] (6.06,0.06) rectangle (6.93,0.93);
\draw[line width=1pt] (0,0) rectangle (7,1);
\draw[line width=1pt] (0,1) rectangle (6,2);\draw[line width=1pt] (0,2) rectangle (5,3);\draw[line width=1pt] (0,3) rectangle (5,4);\draw[line width=1pt] (0,4) rectangle (3,5);\draw[line width=0.95pt] (0,5) rectangle (2,6);\draw[line width=0.95pt] (0,6) rectangle (1,7);
\draw[line width=1pt] (1,0) -- (1,7);\draw[line width=1pt] (2,0) -- (2,6);\draw[line width=1pt] (3,0) -- (3,5);\draw[line width=0.95pt] (4,0) -- (4,4);\draw[line width=1pt] (5,0) --(5,4);\draw[line width=1pt] (6,0) -- (6,2);
\end{tikzpicture}
\begin{tikzpicture}[scale=0.25]
\filldraw[color=red!65] (4.07,2.07)rectangle (4.93,2.93);
\filldraw[color=red!65] (3.07,2.07)rectangle (3.93,2.93);
\filldraw[color=green!45] (3.07,3.07)rectangle (4.93,3.93);
 \filldraw[color=blue!65] (4.07,3.07)rectangle (4.93,3.93); \node at (-1.2,3.5){\scriptsize{r+1}}; \node at (4.5,-0.7){\scriptsize{e+1}};
\filldraw[color=red!65] (0.06,0.06) rectangle (5.93,1.93);
\filldraw[color=red!65] (0.06,2.06) rectangle (2.93,3.93);\filldraw[color=red!65] (0.06,2.06) rectangle (2.93,3.93);
\filldraw[color=red!65] (0.06,4.06) rectangle (2.93,4.93); \filldraw[color=red!65] (0.06,5.06) rectangle (1.93,5.93);
 \filldraw[color=red!65] (0.06,6.06) rectangle (0.93,6.93); \filldraw[color=red!65] (6.06,0.06) rectangle (6.93,0.93);
\draw[line width=1pt] (0,0) rectangle (7,1);
\draw[line width=1pt] (0,1) rectangle (6,2);\draw[line width=1pt] (0,2) rectangle (5,3);\draw[line width=1pt] (0,3) rectangle (5,4);\draw[line width=1pt] (0,4) rectangle (3,5);\draw[line width=0.95pt] (0,5) rectangle (2,6);\draw[line width=0.95pt] (0,6) rectangle (1,7);
\draw[line width=1pt] (1,0) -- (1,7);\draw[line width=1pt] (2,0) -- (2,6);\draw[line width=1pt] (3,0) -- (3,5);\draw[line width=0.95pt] (4,0) -- (4,4);\draw[line width=1pt] (5,0) --(5,4);\draw[line width=1pt] (6,0) -- (6,2);
\end{tikzpicture}
\end{array}$,
 the rows $r$ and $r+1$ of  $\lambda$ have the same length as well as the columns $ e$ and $e+1$.  If $\eta$ and $\eta'$ are obtained by erasing in $\rho$  the green cells $(r,e+1)$ and   $(r+1,e)$ respectively,
 then $F_{\eta}$ is symmetrical in $x_r$ and $x_{r+1}$, and  $F_{\eta'}$ in $y_e$ and $y_{e+1}$.
    Demazure operators $\pi_r^x$ (acting on $x_r$ and $x_{r+1}$) and  $\pi_e^y$ preserve $F_{\eta}$ and $ F_{\eta'} $, 
  and  reproduce  the green cells $(r,e+1)$ and $(r+1,e)$, when acting on $F_{\rho}$,  by creating both the blue cell $(r+1,e+1)$ above $\rho$. That is,
$\scriptsize\pi_r^x F_{\rho}=( \pi_r (1-x_r y_{e+1})^{-1})F_{\eta}=$ 
$F_{\rho}(1-x_{r+1} y_{e+1})^{-1}=F_{\lambda}$ and, similarly, $\pi_{e}^y F_{\rho}=F_{\rho}(1-x_{r+1} y_{e+1})^{-1}=F_{\lambda}$. Henceforth,
$F_{\lambda}=\sum_{\nu\in \mathbb{N}^n}\pi^x_r \widehat{\kappa}_\nu(x) \kappa_{\omega\nu}(y)=\sum_{\nu\in \mathbb{N}^n} \widehat{\kappa}_\nu(x) \pi^y_{e} \kappa_{\omega\nu}(y).$

We give here a combinatorial interpretation of this algebraic explanation for 
\eqref{ext},  when  $\lambda$ is the near staircase  of size $n$:
$\label{nearstair}(\star)~
\begin{tikzpicture} [scale=0.37]
\draw[line width=0.5pt](-0.7,0) -- (-0.7,8);
\draw[line width=0.5pt](-0.7,0) -- (-0.6,0);
\draw[line width=0.5pt](-0.7,8) -- (-0.6,8);
\node at (-1.2,4){\scriptsize n};
\draw[dashed](0,0) rectangle (8,8);
\draw[dashed](0,1.5) -- (8,6.5);
\draw[line width=0.5pt](0,0) rectangle (1,8);
\draw[line width=0.5pt](1,0) rectangle (2,8);
\draw[line width=0.5pt](2,0) rectangle (3,7);
\draw[line width=0.5pt](3,0) rectangle (3,6);
\draw[line width=0.5pt](4,0) rectangle (5,4);
\draw[line width=0.5pt](5,0) rectangle (6,4);
\draw[line width=0.5pt](6,0) rectangle (7,3);
\draw[line width=0.5pt](7,0) rectangle (8,2);
\draw[line width=0.5pt](8,0) rectangle (8,1);
\draw[line width=0.5pt](7,2) -- (7,3);
\draw[line width=0.5pt](8,0) rectangle (8,2);
\draw[line width=0.5pt](8,0) rectangle (8,1);

\draw[line width=0.5pt](3,4) rectangle (4,6);
\draw[line width=0.5pt](2,6) -- (3,6);
\draw[line width=0.5pt](5,3) -- (6,3);

\draw[line width=0.5pt](0,0) -- (8,0);
\draw[line width=0.5pt](0,1) -- (8,1);
\draw[line width=0.5pt](0,2) -- (8,2);
\draw[line width=0.5pt](0,3) -- (6,3);
\draw[line width=0.5pt](0,4) -- (4,4);
\draw[line width=0.5pt](0,5) -- (4,5);
\draw[line width=0.5pt](0,6) -- (2,6);
\draw[line width=0.5pt](0,7) -- (2,7);
\draw[line width=0.5pt](1,7) rectangle (1,8);

\filldraw[color=blue!55] (1.08,7.08)rectangle (1.92,7.92);
\filldraw[color=blue!55] (2.08,6.08)rectangle (2.92,6.92);
\filldraw[color=blue!55] (3.08,5.08)rectangle (3.92,5.92);

\filldraw[color=yellow!95] (5.08,3.099)rectangle (5.92,3.92);
\filldraw[color=yellow!95] (7.08,1.08)rectangle (7.92,1.92);
\filldraw[color=yellow!95] (6.08,2.08)rectangle (6.92,2.92);

\node at (6.1,3.5){\scriptsize $e_{p+1}$};
\node at (6.5,2.5){ \tiny$\ddots$};
\node at (7.5,1.5){\scriptsize $e_k$};
\node at (3.5,5.5){\scriptsize$r_1$};\node at
(2.5,6.5){\tiny$\ddots$};\node at (1.5,7.5){\scriptsize$r_p$};
\draw[line width=0.5pt](1,6) -- (1,8);
\end{tikzpicture}$
\noindent with one layer of $k$ cells, $0\le p\le k<n$,   sited  on the  stairs of $\rho$, 
 avoiding the top and the basement. Given the labels $1\le r_k<\cdots<r_{p+1}<r_1<\cdots<r_p<n$. The  SE labels, $e_{p+1}=n-r_{p+1}<\cdots<e_k=n-r_k$,   NW labels, $r_1<\cdots<r_p$, with respect to the cutting line shown in $(\star)$, indicate  the column indices  $e_j+1$,  $p+1\le j\le k$, and  the row indices  $r_i+1$, $1\le i\le p$, counted in French convention.
In Section \ref{sec:ssaf}, we give the necessary combinatorial tools, being SSAFs the major ones,  the  detectors of keys in \cite{masondemazure}. Section \ref{sec:bruhat} recalls briefly the Bruhat order in $\mathfrak{S}_n$  and  orbits. Sections \ref{sec:rsk} and \ref{sec:crys} are devoted to the  analysis of the behaviour of the \cite{masonrsk}'s RSK analogue under the action of  crystal operators, or coplactic operators, $e_r$, $f_r$ via a growth diagram interpretation: the former delete  convex corners of a Ferrers shape of height or width bigger than one, and thus the later inflate a concave  to a convex corner. Theorems \ref{corcor1}, \ref{bijection} and \ref{bijection-NW-SE}  detect how the key-pair of a  SSYT-pair, of the same shape, change in Bruhat order, when a cell is created in  a concave corner of a Ferrers shape.
Lastly, in Section \ref{sec:last}, the bijective proof of the expansion of $F_\lambda$ over $(\star)$,  as below, is given. As usual, if $m<n$ are in $\mathbb{N}$, put  $[m,n]:=$$\{m,\dots,n\}$, and  $[n]$ for $m=1$. For each $(z,t)\in [0, p]\times [0, k-p]$, let  $(H_z,M_t)\in {[p]\choose z}\times {[p+1,k]\choose t}$, and  $\mathcal{A}_{z,t}^{H_z,M_t}$  the set of  SSAF-pairs with shapes satisfying certain inequalities, in the Bruhat order, as defined in Section \ref{sec:last}. Then

$\label{expansion}\displaystyle{
F_\lambda=\sum_{(F,G)\in \mathcal{A}^{\emptyset,\emptyset}}x^Fy^G
+ \displaystyle\sum_{z=1}^p \sum_{H_z\in{[p]\choose z}}\sum_{\begin{smallmatrix}(F,G)\in \mathcal{A}_{z,0}^{H_z,\emptyset}\end{smallmatrix}}x^Fy^G
+\sum_{t=1}^{k-p}\sum_{M_t \in{[p+1,k]\choose t}}\sum_{\begin{smallmatrix}(F,G)\in \mathcal{A}_{0,t}^{\emptyset, M_t}\end{smallmatrix}}x^Fy^G}$

$+\displaystyle{\sum_{\begin{smallmatrix}(z,t)\in [p]\times [k-p]\end{smallmatrix}}\sum_{(H_z,M_t)}\sum_{\begin{smallmatrix}(F,G)\in \mathcal{A}_{z,t}^{H_z,M_t}\end{smallmatrix}}x^Fy^G
=}$ $\displaystyle{\sum_{\nu\in \mathbb{N}^n}\pi_{r_1}\dots\pi_{r_p}\widehat\kappa_\nu(x)\pi_{e_{p+1}}\dots\pi_{e_k}\kappa_{\omega\nu}(y)} 
$

 $
=\displaystyle{\sum_{\nu\in \mathbb{N}^n}\widehat\kappa_\nu(x)\pi_{n-r_p}\cdots\pi_{n-r_1}\pi_{e_{p+1}}\cdots\pi_{e_k}\kappa_{\omega\nu}(y)=\sum_{\nu\in \mathbb{N}^n}\pi_{n-e_{k}}\cdots\pi_{n-e_{p+1}}\pi_{r_1}\cdots\pi_{r_p}\widehat\kappa_\nu(x)\kappa_{\omega\nu}(y).}$

The two last expansions are the SE, NW versions, with respect to the cutting lines through the top of the first column, and  the end of the botton row  in $(\star)$, respectively. They can be thought as $p=0$ and $p=k$.
\section{SSYTs and  RRSYTs. SSAFs detectors of keys.
}
\label{sec:ssaf}
 A  weak composition $\gamma=(\gamma_1\dots,\gamma_n)$ is a
vector in $\mathbb{N}^n.$
When its entries are in weakly
decreasing order, that is, $\gamma_1 \geq \cdots \geq \gamma_n,$ it is said to be a
 partition.
Every weak composition $\gamma$ determines a unique partition $\lambda$,
 arranging its entries in weakly decreasing
order. It is the unique partition in the orbit of $\gamma$ regarding the usual action of symmetric group $\mathfrak{S}_n$ on $\mathbb{N}^n$. A partition $\lambda=(\lambda_1,\dots, \lambda_n)$ is
identified  with its Young diagram (or Ferrers shape) $dg(\lambda)$ in French
convention, an array of left-justified cells (boxes) with $\lambda_i$ cells
in row $i$ from the bottom, for $1 \leq i \leq n.$  The cells are
located in the diagram  $dg(\lambda)$ by their row and column
indices $(i,j)$, where $1 \leq i \leq n$ and $1 \leq j \leq \lambda_i.$
A filling of shape $ \lambda$ (or a filling of $dg(\lambda)$), in the alphabet $[n]$, is a map $P : dg(\lambda)\rightarrow
 [n].$ A  semi-standard Young tableau (SSYT) $P$ of shape $sh(P)=\lambda$, in the alphabet $[n]$,  is a
filling of $dg(\lambda)$  weakly increasing in each row from
left to right and strictly increasing up in each column.
 The   column word of  the SSYT $P$  is the word consisting
of the entries of each column, read top to bottom and  left
to right. The  content $c(P)=(\alpha_1,\dots,\alpha_n)$
or  weight of $P$ is the content or weight of its column word, that is,
  $\alpha_i$ is the multiplicity of $i\in [n]$ in the column word of $P$.
 A key tableau is a SSYT such that the set of entries
in the $(j + 1)^{th}$ column is a subset of the set of entries in
the $j^{th}$ column, for all $ j$. There is a bijection 
 between weak compositions in $\mathbb{N}^n$ and keys in the alphabet $[n]$ given by $\gamma\rightarrow key(\gamma),$ where
  $key(\gamma)$ is the SSYT such
that for all $j$, the first $\gamma_j$ columns contain the letter
$j$. Any key tableau is of the form $key(\gamma)$ where $\gamma$ is the
content and  the shape is the unique partition in its $\mathfrak{S}_n$-orbit.
A reverse  semi-standard Young tableau (RSSYT), ${R}$, of shape $sh({R})=\lambda$, in the alphabet $[n]$, is a filling of $dg(\lambda)$ such that the entries in each row are weakly decreasing from left to right, and strictly decreasing from bottom to top.
 The reverse Schensted  insertion applied to the word $b=b_1\cdots b_m$, over the alphabet $[n]$, gives a reverse SSYT of the same weight. It  is the same as
 applying the   Schensted  insertion to the  word $b^*=n-b_m+1\cdots n-b_1+1$ to get a SSYT of reverse content, and
then changing   $i$ to $n-i+1$  to obtain  the reverse SSYT $\widetilde P$ with the same weight as $b$. See \cite{fulton}, Appendix A.1, and \cite{stanley}.
Thus  reverse Schensted insertion applied to the column word of a SSYT, defines a weight and shape preserving bijection  between SSYTs and RSSYTs. For instance,
$\label{ex1}\scriptsize\begin{tikzpicture}[scale=0.6]
\draw [line width=1pt] (0,0)rectangle(0.5,0.5);\draw [line width=1pt] (0.5,0)rectangle(1,0.5);
\draw [line width=1pt] (1,0)rectangle(1.5,0.5);\draw [line width=1pt] (1.5,0)rectangle(2,0.5);
\draw [line width=1pt] (0,0.5)rectangle(0.5,1);
\draw [line width=1pt] (0.5,0.5)rectangle(1,1);
\draw [line width=1pt] (0,1)rectangle(0.5,1.5);\node at (1,-0.5){\scriptsize$dg(\lambda)$};
\end{tikzpicture}
~~~~~~~~\begin{tikzpicture}[scale=0.6]
\draw [line width=1pt] (0,0)rectangle(0.5,0.5);\draw [line width=1pt] (0.5,0)rectangle(1,0.5);
\draw [line width=1pt] (1,0)rectangle(1.5,0.5);\draw [line width=1pt] (1.5,0)rectangle(2,0.5);
\draw [line width=1pt] (0,0.5)rectangle(0.5,1);
\draw [line width=1pt] (0.5,0.5)rectangle(1,1);
\draw [line width=1pt] (0,1)rectangle(0.5,1.5);\node at (0.25,0.25){1};\node at (0.75,0.25){2};\node at (1.25,0.25){3};\node at (1.75,0.25){4};\node at (0.25,0.75){2};\node at (0.75,0.75){5};\node at (0.25,1.25){3};\node at (0.75,-0.5){\scriptsize SSYT $P$};
\end{tikzpicture}
~~~~~~~~~~\begin{tikzpicture}[scale=0.6]
\draw [line width=1pt] (0,0)rectangle(0.5,0.5);\draw [line width=1pt] (0.5,0)rectangle(1,0.5);
\draw [line width=1pt] (1,0)rectangle(1.5,0.5);\draw [line width=1pt] (1.5,0)rectangle(2,0.5);
\draw [line width=1pt] (0,0.5)rectangle(0.5,1);
\draw [line width=1pt] (0.5,0.5)rectangle(1,1);
\draw [line width=1pt] (0,1)rectangle(0.5,1.5);\node at (0.25,0.25){5};\node at (0.75,0.25){3};\node at (1.25,0.25){3};\node at (1.75,0.25){2};
\node at (0.25,0.75){4};\node at (0.75,0.75){2};\node at (0.25,1.25){1};\node at (0.75,-0.5){\scriptsize RSSYT ${\scriptsize\widetilde P}$};
\end{tikzpicture}$
are, in this order, the Ferrers diagram of $\lambda=(4,2,1)$, a SSYT  of shape $\lambda$ and  content $(1,2,2,1,1)$, and
 the reverse Schensted insertion  of $P$. 
\subsection{
Weight preserving, shape rearranging bijection between SSYTs and  SSAFs
}
\label{rho}
 A weak composition $\gamma = (\gamma _1,
\dots , \gamma_n)$ is visualised as a diagram consisting of $n$
columns, with $\gamma_j$ cells (boxes) in column $j$, for $1\le j \le n$. The \textit{column
diagram} of $\gamma$ is the set $ dg^{\prime}(\gamma) = \{(i, j)\in
\mathbb{N}^2 : 1 \leq j \leq n, 1 \leq i \leq \gamma_j\}$ where the
coordinates are in French convention, $i$ 
indexing the
rows, and $j$ 
 indexing the columns. (The prime reminds
that the components of $\gamma$ are the column lengths.)
A cell  in a column diagram is written $(i,j),$ where $i$ is the row index
and $j$  the column index. The \textit{augmented diagram} of $\gamma$,
$\widehat{dg}(\gamma)=dg^{\prime}(\gamma)\cup \{(0,j): 1\leq j \leq
n\}$, is the column diagram with $n$ extra cells adjoined in row
$0$. This adjoined row is called the \textit{basement} and it always
contains the numbers $1$ to $n$ in strictly increasing order. The
shape of $\widehat{dg}(\gamma)$ is defined to be $\gamma.$ The empty augmented diagram consists of the basement.
\textit{Semi-skyline augmented fillings} (SSAFs) are  the output of the weight preserving injective map  $\varrho$,   \cite{masondemazure},  acting on RSSYTs as follows.
Let $\widetilde P$ be a RSSYT in the alphabet $[n]$. Define the empty semi-skyline augmented filling as the empty  augmented diagram  with basement elements from $1$ to $n$.
Pick the first column of $\widetilde{P},$ say, $P_1.$
Put all the elements of the first column $P_1$  to the top of the same basement elements in the empty semi-skyline augmented filling. The new diagram is called the semi-skyline augmented filling corresponding to the first column of $\widetilde{P}$ and is denoted by SSAF.
Assume that the first  $i$ columns of $\widetilde{P}$, denoted $P_1, P_2, \dots, P_i$, have been mapped to a SSAF. Consider the largest element, $a_1$, in
the $(i + 1)$-th column $P_{i+1}$. There exists an element greater than or equal to $a_1$ in the i-th row
of the SSAF. Place $a_1$ on top of the
leftmost such element.
Assume that the largest $k-1$ entries in $P_{i+1}$ have been placed into the SSAF. The $k$-th
largest element, $a_k$, of $P_{i+1}$ is then placed into the SSAF. Place $a_k$ on top of the leftmost entry $b$ in row $k-1$ such that $b \ge a_k$ and the cell immediately above $b$ is empty.  Continue this procedure until all entries in
$P_{i+1}$ have been mapped into the $(i + 1)$-th row and then repeat for the remaining columns of $\widetilde{P}$
to obtain the semi-skyline augmented filling $F$.
  The content of the SSAF $F$  is the vector $c(F):=c(\widetilde P)\in \mathbb{N}^n$, and the shape of $F$,   $sh(F)$, is  the weak composition  recording  the length of the columns of $F$, from left to right. Hence  a rearranging of $sh(\tilde P)$.  
Thereby,  the  reverse Schensted insertion composed with $\varrho$
 is a weight preserving and shape rearranging bijection, denoted  $\Psi$, between SSYTs and SSAFs. \cite{masondemazure} proves that if $P$ is a SSYT,
  the right key of $P$, 
a notion  due to   \cite{lascouxkeys}, is 
$key(sh(\Psi(P)))=key(sh(\varrho(\widetilde P))).$  We just say, indistinctly, that $sh(F)$ is the (right)  key of $F$ or $P$.
For instance,
$\label{ex3} \begin{tikzpicture}[scale=0.6]
\draw [line width=1pt] (-4,0)rectangle(-3.5,0.5);\draw [line width=1pt]
(-3.5,0)rectangle(-3,0.5);
\draw [line width=1pt] (-3,0)rectangle(-2.5,0.5);\draw [line width=1pt]
(-2.5,0)rectangle(-2,0.5);
\draw [line width=1pt] (-4,0.5)rectangle(-3.5,1);
\draw [line width=1pt] (-3.5,0.5)rectangle(-3,1);
\draw [line width=1pt] (-4,1)rectangle(-3.5,1.5);\node at
(-3.75,0.25){\scriptsize 1};\node at (-3.25,0.25){\scriptsize 2};\node at (-2.75,0.25){\scriptsize 3};\node at
(-2.25,0.25){\scriptsize 4};\node at (-3.75,0.75){\scriptsize 2};\node at (-3.25,0.75){\scriptsize 5};\node at
(-3.75,1.25){\scriptsize 3};\node at (-5,0){ \scriptsize $P=$};
\end{tikzpicture}$
$\begin{tikzpicture}[scale=0.6]
\draw [line width=1pt] (-5.7,-2)--(-5.3,-2);\node at (-5,-1.5){$$};
\draw [line width=1pt] (-4,-2)rectangle(-3.5,-1.5);\draw [line width=1pt]
(-3.5,-2)rectangle(-3,-1.5);
\draw [line width=1pt] (-3,-2)rectangle(-2.5,-1.5);\draw [line width=1pt]
(-2.5,-2)rectangle(-2,-1.5);
\draw [line width=1pt] (-4,-1.5)rectangle(-3.5,-1);
\draw [line width=1pt] (-3.5,-1.5)rectangle(-3,-1);
\draw [line width=1pt] (-4,-1)rectangle(-3.5,-0.5);\node at
(-3.75,-1.75){\scriptsize 5};\node at (-3.25,-1.75){\scriptsize 3};\node at (-2.75,-1.75){\scriptsize 3};\node at
(-2.25,-1.75){\scriptsize 2};\node at (-3.75,-1.25){\scriptsize 4};\node at (-3.25,-1.25){\scriptsize 2};\node at
(-3.75,-0.75){\scriptsize 1};\node at (-4.75,-2){\scriptsize $\widetilde{P}=$};
\end{tikzpicture}$
$\begin{tikzpicture}[scale=0.6]
\draw [line width=1pt] (-1.5,0)--(-0.5,0);\node at (-1,0.5){$\varrho$};
\draw [line width=1pt] (0,0)--(2.5,0);\node at (0.25,-0.25){\scriptsize 1};\node at
(0.75,-0.25){\scriptsize 2};\node at (1.25,-0.25){\scriptsize 3};\node at (1.75,-0.25){\scriptsize 4};\node at
(2.25,-0.25){\scriptsize 5};
\draw [line width=1pt] (0,0)rectangle(0.5,0.5);\draw [line width=1pt]
(1.5,0)rectangle(2,0.5);\draw [line width=1pt]
(1.5,0.5)rectangle(2,1);\draw [line width=1pt]
(1.5,1)rectangle(2,1.5);\draw [line width=1pt]
(1.5,1.5)rectangle(2,2);\draw [line width=1pt]
(2,0)rectangle(2.5,0.5);\draw [line width=1pt] (2,0.5)rectangle(2.5,1);
\node at (0.25,0.25){\scriptsize 1};\node at (1.75,0.25){\scriptsize 4};\node at
(1.75,0.75){\scriptsize 3};\node at (1.75,1.25){\scriptsize 3};\node at (1.75,1.75){\scriptsize 2};\node at
(2.25,0.25){\scriptsize 5};\node at (2.25,0.75){\scriptsize 2};\node at (4.2,0.1){\scriptsize$=\Psi(P)=F$};
\end{tikzpicture}$,
with $c(F)=(1,2,2,1,1)=c(\widetilde P)=c(P)$,  $sh(F)=(1,0,0,4,2)$ and $key(1,0,0,4,2)$.

\noindent{\bf Remark $1$}.  {\em  If, in  $F$,  the height of column i is $<$ than the height of column $j$,
$j>i$, the choice of  the left most position  implies that $a < b \le  c$  in
any triple $\begin{array}{ccc}
 &b\\
a\;\cdots& c
\end{array}$ with $b$, $c$  in column $j$, and $a$ in column $i$}.

\section{
 RSK analogue,  growth diagram of RRSK and  the key-pair}
\label{sec:rsk}
 The two line array
$w=\scriptsize\left(\begin{array}{cccc}j_1&j_2&\cdots&j_l\\i_1&i_2&\cdots&i_l
\end{array}\right)$, such that $j_{r}< j_{r+1}$, and if $j_r=j_{r+1}$ then $i_{r} \le i_{r+1}$, for all $1\le r\le l-1$,
where $ i_r, ~j_r\in [n]$,
is called a biword in lexicographic order, on the alphabet $[n]$, with respect to the first row.
The RSK algorithm is a bijection between the
biwords in lexicographic order on the alphabet $[n]$, and the pairs of SSYTs of the same shape on the same alphabet.
The reverse RSK (RRSK) algorithm, see \cite{stanley}, is the same as applying RSK  to the biword $\scriptsize{w}^*=\left(\begin{array}{cccc}n-j_l+1&\dots& n-j_1+1\\n-i_l+1&\cdots&n-i_1+1\end{array}\right)$,
 to get a pair $(P,Q)$ of SSYTs, and then change $i$ to $n-i+1$, in all their entries, to obtain a pair $(\widetilde{P},\widetilde{Q})$ of reverse SSYTs.
 \cite{masonrsk} uses an  analogue of Schensted insertion
 to find an analogue
$\Phi$ of the RSK to produce pairs of SSAFs.
 {\em The map $\Phi$ is a bijection between
the   biwords  in lexicographic order in
the alphabet $[n]$, and the  pairs  of SSAFs with shapes (keys)  in some $\mathfrak{S}_n$-orbit.}
The bijection $\Phi$ applied to a biword $w$ is the same as applying the RRSK to $w$ and then applying $\varrho$ to each reverse SSYT of the output pair $(\widetilde P,\widetilde Q)$. That is, $\Phi(w)=(\varrho(\widetilde P),\varrho(\widetilde Q))$. Equivalently,  applying the RSK to $w$ and then $\Psi$ to each  SSYT of the output pair $( P, Q)$ gives  $\Phi(w)=(\Psi(P),\Psi(Q))$.

The remaining  of this section  follows  closely  \cite{kratt} and \cite{stanley}.
Let $w$ be a biword in the lexicographic order, over the alphabet $[n]$, represented  in the $n\times n$ square diagram, by putting the number  $r$  in the cell $(i,j)$  of the  square grid, whenever the biletter $\binom{j}{i}$ 
  appears $r\ge 1$ times in the biword $w$. (Rows are counted from bottom to top and  columns from left to right.)
 Scanning the columns, from left to right  and bottom to top, the biword $w$ is recovered in lexicographic order.
The 01-filling of this diagram 
 has at most one $1$ in each row and each column as follows.
Construct a rectangle diagram with more rows and columns as follows.
The entries which are originally in the same column or in the same row are put in
different columns and rows in the larger diagram. An entry $m$ is  replaced by
$m$ $1$'s in the new diagram, all of them placed in different rows and columns. The entries in a row are separated from bottom/left to top/right, and  the 1's are represented by ${\bf X}$'s.
    If there should be several entries in a column as well,  separate entries in a column from bottom/left to top/right. In the cell with entry $m$, we replace $m$ by a chain of $m$ ${\bf X}$'s arranged from bottom/left to top/right.
 The original $n$ columns and $n$ rows are indicated by thick lines,
whereas the newly created columns and rows are indicated by thin lines.
To give an interpretation of RRSK in terms of growth diagrams, we start by
assigning the empty partition $\emptyset$ to each point of a corner cell on the right column and on the top row of the $01$-filling. Then
assign  partitions to the other corners inductively by applying the following  local rules.
Consider the cell
$\begin{tikzpicture}[scale=0.2]
\draw[line width=0.5pt] (0,0) rectangle
(2,2);
\node at (2.3,2.3){$\varepsilon$};\node at (2.3,-0.3){$\mu$};\node at (-0.3,2.2){$\nu$};
\end{tikzpicture}$
labeled by the partitions  $\varepsilon, \mu, \nu$, such that $\varepsilon \subseteq \mu$ and $\varepsilon \subseteq \nu$, where  the containment means that the Ferrres shapes differ at most by one box. Then $\beta$ is determined as follows:
(i) If $\varepsilon=\mu=\nu,$ and if there is no cross in the cell, then
$\beta=\varepsilon $. (ii)  If $\varepsilon=\mu\neq\nu,$ then $\beta=\nu.$
 If $\varepsilon=\nu\neq\mu,$ then $\beta=\mu.$
(iii)  If $\varepsilon, \mu, \nu$ are pairwise different, then
$\beta=\mu\cup\nu$, \textit{i.e}, $\beta_i=max\{\mu_i,\nu_i\}$.
(iv) If $\varepsilon\neq\mu=\nu,$ then $\beta$ is formed by adding a
box to the $(k +1)$-st row of $\mu=\nu,$ given that $\mu=\nu$ and
$\varepsilon$ differ in the $k$-th row.
(v)  If $\varepsilon=\mu=\nu,$ and if there is a cross in the cell, then
$\beta$ is formed by adding a box to the first row of
$\varepsilon=\mu=\nu.$

Applying the local rules leads to a pair of nested sequences of partitions on the left column and in the bottom row of the growth diagram.
 Let $\beta^i$ be the partition assigned to the $i$-th  thick column,  on the bottom row of the growth diagram, when we scan the thick columns from  right  to left, with the rightmost column being column $n$.
 Then the bottom row labelling, assigned to the thick columns of the
 growth diagram, produces a sequence of partitions $ \underline\beta^0 \supseteq\cdots \supseteq  \underline\beta^{n-1}\supseteq \underline\beta^n=\emptyset$,
 such that  $\underline\beta^{i-1}/\underline\beta^{i}$ is a horizontal strip.
Let $\beta^i$ be the partition assigned to the $i$-th  thick row, on the left of the growth diagram, when we scan the thick rows  from  top to bottom, with the top row being row $n$.
 Then the left column labelling, assigned to the thick rows of the
 growth diagram, produces a sequence of partitions $\emptyset = \beta^n\subseteq\beta^{n-1}\subseteq\cdots \subseteq\beta^0$,
such that
 $\beta^{i-1}/\beta^{i}$ is a horizontal strip.
 Filling in, with  $i$,   the cells of $\beta^{i-1}/\beta^{i}$ and  $\underline\beta^{i-1}/\underline\beta^{i}$, for $i\ge 1$, produces, in this order, the pair $(\widetilde{P},\widetilde{Q})$ of RSSYTs of the same shape.
 This is the same as applying the reverse RSK to the biword $w$.
For instance, if $w=\scriptsize\left(\begin{array}{ccccccccc}1&1&2&3&4&4&5&7&7\\2&7&2&4&1&3&3&1&1\end{array}\right)$,
$\scriptsize\begin{array}{ccccccc}
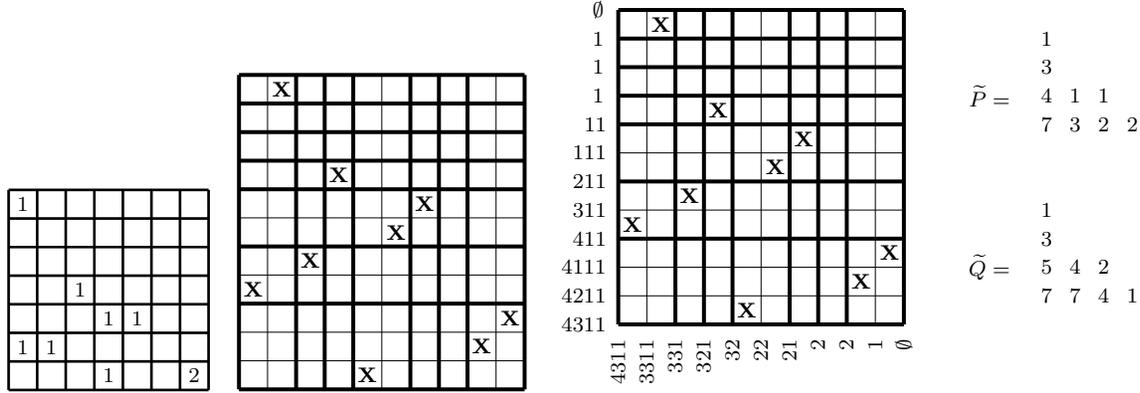
\label{fig:biword}\begin{tikzpicture}[scale=0.38] \draw[line width=1pt] (0,0) --
(7,0);\draw[line width=1pt] (0,0) -- (0,7); \draw[line width=1pt]
(0,1) -- (7,1);\draw[line width=1pt] (0,2) -- (7,2);\draw[line
width=1pt] (0,3)--  (7,3);\draw[line width=1pt] (0,4) --
(7,4);\draw[line width=0.95pt] (0,5) -- (7,5);\draw[line width=1pt]
(0,6) -- (7,6);\draw[line width=1pt] (0,7) -- (7,7); \draw[line
width=1pt] (1,0) -- (1,7);\draw[line width=1pt] (2,0) --
(2,7);\draw[line width=0.95pt] (3,0) -- (3,7);\draw[line width=1pt]
(4,0) -- (4,7);\draw[line width=1pt] (5,0) --(5,7);\draw[line
width=1pt] (6,0) -- (6,7);\draw[line width=1pt] (7,0) -- (7,7);
\node at (0.5,1.5){1};\node at (0.5,6.5){1};\node at
(1.5,1.5){1};\node at (2.5,3.5){1};\node at (3.5,0.5){1};\node at
(3.5,2.5){1};\node at (4.5,2.5){1};\node at (6.5,0.5){2};
\end{tikzpicture}&
\label{fig:01}\scriptsize\begin{tikzpicture}[scale=0.38]
\draw[line width=1.5pt] (0,0) -- (0,11);\draw[line width=0.25pt] (1,0)
-- (1,11);\draw[line width=1.5pt] (2,0) -- (2,11);\draw[line
width=1.5pt] (3,0) -- (3,11);\draw[line width=1.5pt] (4,0) --
(4,11);\draw[line width=0.25pt] (5,0) -- (5,11);\draw[line width=1.5pt]
(6,0) -- (6,11);\draw[line width=1.5pt] (7,0) -- (7,11);\draw[line
width=1.5pt] (8,0) -- (8,11);\draw[line width=0.25pt] (9,0) --
(9,11);\draw[line width=1.5pt] (10,0) -- (10,11);\draw[line width=1.5pt] (0,0) --
(10,0);\draw[line width=0.25pt] (0,1) -- (10,1);\draw[line
width=0.25pt] (0,2) -- (10,2);\draw[line width=1.5pt] (0,3) --
(10,3);\draw[line width=0.25pt] (0,4) -- (10,4); \draw[line
width=1.5pt] (0,5) -- (10,5);\draw[line width=0.25pt] (0,6) --
(10,6);\draw[line width=1.5pt] (0,7) -- (10,7);\draw[line width=1.5pt]
(0,8) -- (10,8); \draw[line width=1.5pt] (0,9) -- (10,9);\draw[line
width=1.5pt] (0,10) -- (10,10);\draw[line width=1.5pt] (0,11) --
(10,11);\node at (0.5,3.5){{\bf X}};\node at (
1.5,10.5){{\bf X}};\node at (2.5,4.5){{\bf X}};\node at
(3.5,7.5){{\bf X}};\node at (4.5,0.5){{\bf X}};\node at
(5.5,5.5){{\bf X}};\node at (6.5,6.5){{\bf X}};\node at
(8.5,1.5){{\bf X}};\node at (9.5,2.5){{\bf X}};
\end{tikzpicture}
&\scriptsize\begin{tikzpicture}[scale=0.38] \draw[line width=1.5pt] (0,0) --
(0,11);\draw[line width=0.25pt] (1,0) -- (1,11);\draw[line width=1.5pt]
(2,0) -- (2,11);\draw[line width=1.5pt] (3,0) -- (3,11);\draw[line
width=1.5pt] (4,0) -- (4,11);\draw[line width=0.25pt] (5,0) --
(5,11);\draw[line width=1.5pt] (6,0) -- (6,11);\draw[line width=1.5pt]
(7,0) -- (7,11);\draw[line width=1.5pt] (8,0) -- (8,11);\draw[line
width=0.25pt] (9,0) -- (9,11);\draw[line width=1.5pt] (10,0) --
(10,11);\draw[line
width=1.5pt] (0,0) -- (10,0);\draw[line width=0.25pt] (0,1) --
(10,1);\draw[line width=0.25pt] (0,2) -- (10,2);\draw[line width=1.5pt]
(0,3) -- (10,3);\draw[line width=0.25pt] (0,4) -- (10,4); \draw[line
width=1.5pt] (0,5) -- (10,5);\draw[line width=0.25pt] (0,6) --
(10,6);\draw[line width=1.5pt] (0,7) -- (10,7);\draw[line width=1.5pt]
(0,8) -- (10,8); \draw[line width=1.5pt] (0,9) -- (10,9);\draw[line
width=1.5pt] (0,10) -- (10,10);\draw[line width=1.5pt] (0,11) --
(10,11);\node at (0.5,3.5){{\bf X}};\node at (
1.5,10.5){{\bf X}};\node at (2.5,4.5){{\bf X}};\node at
(3.5,7.5){{\bf X}};\node at (4.5,0.5){{\bf X}};\node at
(5.5,5.5){{\bf X}};\node at (6.5,6.5){{\bf X}};\node at
(8.5,1.5){{\bf X}};\node at (9.5,2.5){{\bf X}}; \node at
(-0.7,10){\scriptsize 1};\node at (-0.7,9){\scriptsize 1};\node at (-0.7,8){\scriptsize1};\node at
(-0.8,7){\scriptsize 11};\node at (-1,6){\scriptsize 111};\node at (-1,5){\scriptsize 211};\node at
(-1,4){\scriptsize 311};\node at (-1,3){\scriptsize 411};\node at
(-1.2,2){\scriptsize 4111};\node at (-1.2,1){\scriptsize 4211};\node at (-1.2,0){\scriptsize 4311};
\node at (9,-0.7)  [ rotate=90] {\scriptsize 1};\node at (8,-0.7) [
rotate=90]{\scriptsize 2};\node at (7,-0.7) [ rotate=90]{\scriptsize 2};\node at (6,-0.9) [
rotate=90]{\scriptsize 21};\node at (5,-0.9) [ rotate=90]{\scriptsize 22};\node at (4,-0.9) [
rotate=90]{\scriptsize 32};\node at (3,-1.1) [ rotate=90]{\scriptsize 321};\node at
(2,-1.1) [ rotate=90]{\scriptsize 331};\node at (1,-1.3) [
rotate=90]{\scriptsize 3311};\node at (0,-1.3) [ rotate=90]{\scriptsize 4311}; \node at
(-0.7,11) {\scriptsize $\emptyset$};\node at (10,-0.7) [
rotate=90]{\scriptsize $\emptyset$};
 \node at (15,7){7};\node at (16,7){3};
\node at (17,7){2};\node at (18,7){2};
\node at (15,8){4};\node at(16,8){1};\node at (17,8){1};
\node at (15,9){3};\node at (15,10){1};
\node at (13,8){$\widetilde{P}=$};
 \node at (15,1){7};\node at
(16,1){7};\node at (17,1){4};\node at (18,1){1};\node at
(15,2){5};\node at (16,2){4};\node at (17,2){2};\node at
(15,3){3};\node at (15,4){1};\node at (13,2){$\widetilde{Q}=$};
\end{tikzpicture}
\end{array}$\\

\noindent with $n=7$,  one has the $7\times 7$ square diagram on the left, the 01-filling in the middle, and the RRSK growth diagram on the right where $(\widetilde P,\widetilde Q)$ is the image of $w$,
The map $\varrho$, defined in Section \ref{rho}, allows to find the pair of SSAFs from the growth diagram   of the reverse RSK.
  Consider   the growth diagram bottom labelling, $\underline \beta^0 \supseteq\cdots \supseteq\underline \beta^{n-1 }\supseteq \underline\beta^n=\emptyset$,  assigned to the thick columns.
 For each $i= 1,\dots,n$, let 
 $\underline\beta^{i_{l_i-1}}\supseteq\cdots\supseteq\underline\beta^{i_1}$, with $\underline\beta^{i_{1}-1}:=\underline\beta^i$, be the bottom sequence of partitions labelling
the  $l_i-1$ thin columns,  strictly in between  the two thick columns  $i-1$ and $i$. 
 \textit{  Start with the empty partition $\underline\beta^n=\emptyset$ 
 and   the empty SSAF with basement $[n]$. Proceed to the left along the growth diagram bottom row labelling.
When we arrive to the partition $\beta^{i_j}$, we put a cell, filled with  $i$, in the leftmost possible  place of the   SSAF such that the shape of the new SSAF is a rearrangement of the partition $\beta^{i_j}$ and the decreasing property on the columns of the SSAF, from the bottom to the top, is preserved. At the end of the scanning of the bottom row labelling, the SSAF $G$ is obtained}. Its shape is a rearrangement of the shape of $\tilde Q$.
 Similarly,  consider   the left side labelling, $\emptyset = \beta^n \subseteq \beta^{n-1} \subseteq \dots \subseteq \beta^0$,  assigned to the thick rows of the
 growth diagram.
 For each $i= 1,\dots,n$, let
 $\beta^{i_1}\subseteq\cdots\subseteq\beta^{i_{l_i-1}}$, with $\beta^{i_{1}-1}:=\beta^i$, be the sequence of partitions labelling 
the  $l_i-1$ thin rows,  strictly in between  the two thick rows  $i-1$ and $i$. At the end of the procedure, when the scanning of the left side labelling is finished, the SSAF $F$ is obtained. Its shape is a rearrangement of the shape of $\tilde P$. See the example below corresponding to the RRSK above. 

$\scriptsize
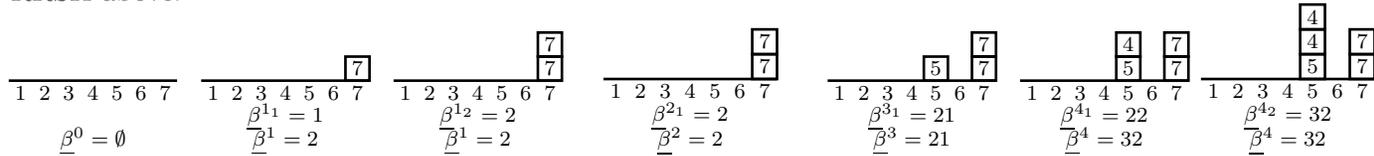
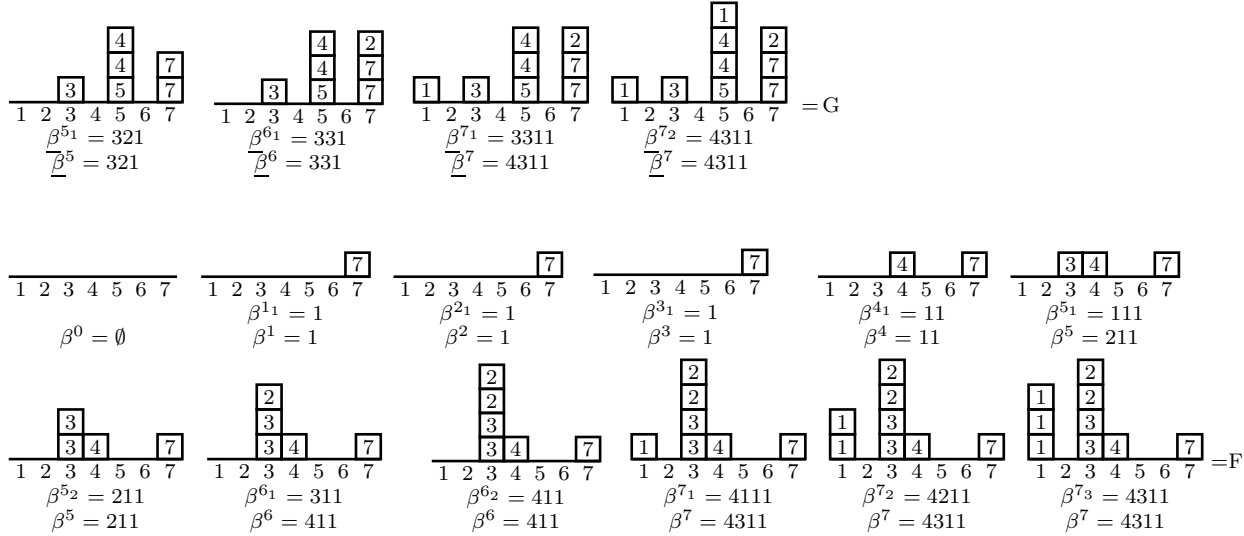
\begin{tikzpicture}[scale=0.32]
\draw[line width=1pt] (0,0.5) --(7,0.5); \draw[line width=1pt] (8,0.5) --(15,0.5); \draw[line width=1pt] (16,0.5) --(23,0.5);
\draw[line width=1pt] (14,0.5) rectangle (15,1.5); \node at (14.5,1){7};\draw[line width=1pt] (22,0.5) rectangle(23,1.5);\node at (22.5,1){7};\draw[line width=1pt] (22,1.5) rectangle(23,2.5);\node at (22.5,2){7};
\node at (0.5,0){1};\node at (1.5,0){2};\node at (2.5,0){3};\node at (3.5,0){4};\node at (4.5,0){5};\node at (5.5,0){6};\node at (6.5,0){7};\node at (8.5,0){1};\node at (9.5,0){2};\node at (10.5,0){3};\node at (11.5,0){4};\node at (12.5,0){5};\node at (13.5,0){6};\node at (14.5,0){7};\node at (16.5,0){1};\node at (17.5,0){2};\node at (18.5,0){3};\node at (19.5,0){4};\node at (20.5,0){5};\node at (21.5,0){6};\node at (22.5,0){7};

\node at (3.5,-1){};\node at (3.5,-2){\scriptsize$\underline\beta^{0}=\emptyset$};
\node at (11.5,-1){\scriptsize$\underline\beta^{1_1}=1$};\node at (11.5,-2){\scriptsize$\underline\beta^{1}=2$};
\node at (19.5,-1){\scriptsize$\underline\beta^{1_2}=2$};\node at (19.5,-2){$\underline\beta^{1}=2$};
\end{tikzpicture}$$~~~~\scriptsize\begin{tikzpicture}[scale=0.33]
\draw[line width=1pt] (0,0.5) --(7,0.5);
\draw[line width=1pt] (6,0.5) rectangle (7,1.5);\node at (6.5,1){7};\draw[line width=1pt] (6,1.5) rectangle (7,2.5);\node at (6.5,2){7};

\node at (0.5,0){1};\node at (1.5,0){2};\node at (2.5,0){3};\node at (3.5,0){4};\node at (4.5,0){5};\node at (5.5,0){6};\node at (6.5,0){7};

\node at (3.5,-1){\scriptsize$\underline\beta^{2_1}=2$};\node at (3.5,-2){\scriptsize$\underline\beta^{2}=2$};
\end{tikzpicture}$
$~~~~\scriptsize\begin{tikzpicture}[scale=0.32]
 \draw[line width=1pt] (8,0.5) --(15,0.5); \draw[line width=1pt] (16,0.5) --(23,0.5);

\draw[line width=1pt] (12,0.5) rectangle (13,1.5); \node at (12.5,1){5};
\draw[line width=1pt] (14,0.5) rectangle (15,1.5); \node at (14.5,1){7};
\draw[line width=1pt] (14,1.5) rectangle (15,2.5); \node at (14.5,2){7};

\draw[line width=1pt] (20,0.5) rectangle(21,1.5);\node at (20.5,1){5};
\draw[line width=1pt] (20,1.5) rectangle(21,2.5);\node at (20.5,2){4};
\draw[line width=1pt] (22,0.5) rectangle(23,1.5);\node at (22.5,1){7};
\draw[line width=1pt] (22,1.5) rectangle(23,2.5);\node at (22.5,2){7};
\node at (8.5,0){1};\node at (9.5,0){2};\node at (10.5,0){3};\node at (11.5,0){4};\node at (12.5,0){5};\node at (13.5,0){6};\node at (14.5,0){7};\node at (16.5,0){1};\node at (17.5,0){2};\node at (18.5,0){3};\node at (19.5,0){4};\node at (20.5,0){5};\node at (21.5,0){6};\node at (22.5,0){7};

\node at (11.5,-1){\scriptsize$\underline\beta^{3_1}=21$};\node at (11.5,-2){\scriptsize$\underline\beta^{3}=21$};
\node at (19.5,-1){\scriptsize$\underline\beta^{4_1}=22$};\node at (19.5,-2){\scriptsize$\underline\beta^{4}=32$};
\end{tikzpicture}$
$\scriptsize\begin{tikzpicture}[scale=0.33]
\draw[line width=1pt] (0,0.5) --(7,0.5);

\draw[line width=1pt] (4,0.5) rectangle (5,1.5);\node at (4.5,1){5};\draw[line width=1pt] (4,1.5) rectangle (5,2.5);\node at (4.5,2){4};\draw[line width=1pt] (4,2.5) rectangle (5,3.5);\node at (4.5,3){4};
\draw[line width=1pt] (6,0.5) rectangle (7,1.5);\node at (6.5,1){7};\draw[line width=1pt] (6,1.5) rectangle (7,2.5);\node at (6.5,2){7};

\node at (0.5,0){1};\node at (1.5,0){2};\node at (2.5,0){3};\node at (3.5,0){4};\node at (4.5,0){5};\node at (5.5,0){6};\node at (6.5,0){7};
\node at (3.5,-1){\scriptsize$\underline\beta^{4_2}=32$};\node at (3.5,-2){\scriptsize$\underline\beta^{4}=32$};

\end{tikzpicture}$

\vskip0.8cm

$\scriptsize\begin{tikzpicture}[scale=0.33]
\draw[line width=1pt] (8,0.5) --(15,0.5);
\draw[line width=1pt] (10,0.5) rectangle (11,1.5); \node at (10.5,1){3};
\draw[line width=1pt] (12,0.5) rectangle (13,1.5); \node at (12.5,1){5};
\draw[line width=1pt] (12,1.5) rectangle (13,2.5); \node at (12.5,2){4};
\draw[line width=1pt] (12,2.5) rectangle (13,3.5); \node at (12.5,3){4};
\draw[line width=1pt] (14,0.5) rectangle (15,1.5); \node at (14.5,1){7};
\draw[line width=1pt] (14,1.5) rectangle (15,2.5); \node at (14.5,2){7};

\node at (8.5,0){1};\node at (9.5,0){2};\node at (10.5,0){3};\node at (11.5,0){4};\node at (12.5,0){5};\node at (13.5,0){6};\node at (14.5,0){7};
\node at (11.5,-1){\scriptsize$\underline\beta^{5_1}=321$};\node at (11.5,-2){\scriptsize$\underline\beta^{5}=321$};

\end{tikzpicture}$
$~~\scriptsize\begin{tikzpicture}[scale=0.32]
  \draw[line width=1pt] (16,0.5) --(23,0.5);
\draw[line width=1pt] (18,0.5) rectangle(19,1.5);\node at (18.5,1){3};
\draw[line width=1pt] (20,0.5) rectangle(21,1.5);\node at (20.5,1){5};
\draw[line width=1pt] (20,1.5) rectangle(21,2.5);\node at (20.5,2){4};
\draw[line width=1pt] (20,2.5) rectangle(21,3.5);\node at (20.5,3){4};
\draw[line width=1pt] (22,0.5) rectangle(23,1.5);\node at (22.5,1){7};
\draw[line width=1pt] (22,1.5) rectangle(23,2.5);\node at (22.5,2){7};
\draw[line width=1pt] (22,2.5) rectangle(23,3.5);\node at (22.5,3){2};
\node at (16.5,0){1};\node at (17.5,0){2};\node at (18.5,0){3};\node at (19.5,0){4};\node at (20.5,0){5};\node at (21.5,0){6};\node at (22.5,0){7};

\node at (19.5,-1){\scriptsize$\underline\beta^{6_1}=331$};\node at (19.5,-2){\scriptsize$\underline\beta^{6}=331$};
\end{tikzpicture}$$~~~\scriptsize\begin{tikzpicture}[scale=0.33]
\draw[line width=1pt] (0,0.5) --(7,0.5); \draw[line width=1pt] (8,0.5) --(15,0.5);
\draw[line width=1pt] (0,0.5) rectangle (1,1.5);\node at (0.5,1){1};
\draw[line width=1pt] (2,0.5) rectangle (3,1.5);\node at (2.5,1){3};
\draw[line width=1pt] (4,0.5) rectangle (5,1.5);\node at (4.5,1){5};\draw[line width=1pt] (4,1.5) rectangle (5,2.5);\node at (4.5,2){4};\draw[line width=1pt] (4,2.5) rectangle (5,3.5);\node at (4.5,3){4};
\draw[line width=1pt] (6,0.5) rectangle (7,1.5);\node at (6.5,1){7};\draw[line width=1pt] (6,1.5) rectangle (7,2.5);\node at (6.5,2){7};\draw[line width=1pt] (6,2.5) rectangle (7,3.5);\node at (6.5,3){2};

\draw[line width=1pt] (8,0.5) rectangle (9,1.5); \node at (8.5,1){1};
\draw[line width=1pt] (10,0.5) rectangle (11,1.5); \node at (10.5,1){3};
\draw[line width=1pt] (12,0.5) rectangle (13,1.5); \node at (12.5,1){5};
\draw[line width=1pt] (12,1.5) rectangle (13,2.5); \node at (12.5,2){4};
\draw[line width=1pt] (12,2.5) rectangle (13,3.5); \node at (12.5,3){4};
\draw[line width=1pt] (12,3.5) rectangle (13,4.5); \node at (12.5,4){1};
\draw[line width=1pt] (14,0.5) rectangle (15,1.5); \node at (14.5,1){7};
\draw[line width=1pt] (14,1.5) rectangle (15,2.5); \node at (14.5,2){7};
\draw[line width=1pt] (14,2.5) rectangle (15,3.5); \node at (14.5,3){2};

\node at (0.5,0){1};\node at (1.5,0){2};\node at (2.5,0){3};\node at (3.5,0){4};\node at (4.5,0){5};\node at (5.5,0){6};\node at (6.5,0){7};\node at (8.5,0){1};\node at (9.5,0){2};\node at (10.5,0){3};\node at (11.5,0){4};\node at (12.5,0){5};\node at (13.5,0){6};\node at (14.5,0){7}; \node at (16.4,0.4){=\,G};

\node at (3.5,-1){\scriptsize$\underline\beta^{7_1}=3311$};\node at (3.5,-2){\scriptsize$\underline\beta^{7}=4311$};
\node at (11.5,-1){\scriptsize$\underline\beta^{7_2}=4311$};\node at (11.5,-2){\scriptsize$\underline\beta^{7}=4311$};
\end{tikzpicture}$

\vskip0.8cm
$\scriptsize\begin{tikzpicture}[scale=0.32]
\draw[line width=1pt] (0,0.5) --(7,0.5); \draw[line width=1pt] (8,0.5) --(15,0.5); \draw[line width=1pt] (16,0.5) --(23,0.5);
\draw[line width=1pt] (14,0.5) rectangle (15,1.5); \node at (14.5,1){7};\draw[line width=1pt] (22,0.5) rectangle(23,1.5);\node at (22.5,1){7};
\node at (0.5,0){1};\node at (1.5,0){2};\node at (2.5,0){3};\node at (3.5,0){4};\node at (4.5,0){5};\node at (5.5,0){6};\node at (6.5,0){7};\node at (8.5,0){1};\node at (9.5,0){2};\node at (10.5,0){3};\node at (11.5,0){4};\node at (12.5,0){5};\node at (13.5,0){6};\node at (14.5,0){7};\node at (16.5,0){1};\node at (17.5,0){2};\node at (18.5,0){3};\node at (19.5,0){4};\node at (20.5,0){5};\node at (21.5,0){6};\node at (22.5,0){7};

\node at (3.5,-1){};\node at (3.5,-2){$\beta^{0}=\emptyset$};
\node at (11.5,-1){$\beta^{1_1}=1$};\node at (11.5,-2){$\beta^{1}=1$};
\node at (19.5,-1){$\beta^{2_1}=1$};\node at (19.5,-2){$\beta^{2}=1$};
\end{tikzpicture}$$~~~\scriptsize\begin{tikzpicture}[scale=0.33]
\draw[line width=1pt] (0,0.5) --(7,0.5);
\draw[line width=1pt] (6,0.5) rectangle (7,1.5);\node at (6.5,1){7};
\node at (0.5,0){1};\node at (1.5,0){2};\node at (2.5,0){3};\node at (3.5,0){4};\node at (4.5,0){5};\node at (5.5,0){6};\node at (6.5,0){7};
\node at (3.5,-1){\scriptsize$\beta^{3_1}=1$};\node at (3.5,-2){\scriptsize$\beta^{3}=1$};

\end{tikzpicture}$
$~~~~\scriptsize\begin{tikzpicture}[scale=0.32]
 \draw[line width=1pt] (8,0.5) --(15,0.5); \draw[line width=1pt] (16,0.5) --(23,0.5);


\draw[line width=1pt] (11,0.5) rectangle (12,1.5); \node at (11.5,1){4};
\draw[line width=1pt] (14,0.5) rectangle (15,1.5); \node at (14.5,1){7};

\draw[line width=1pt] (19,0.5) rectangle(20,1.5);\node at (19.5,1){4};
\draw[line width=1pt] (18,0.5) rectangle(19,1.5);\node at (18.5,1){3};
\draw[line width=1pt] (22,0.5) rectangle(23,1.5);\node at (22.5,1){7};
\node at (8.5,0){1};\node at (9.5,0){2};\node at (10.5,0){3};\node at (11.5,0){4};\node at (12.5,0){5};\node at (13.5,0){6};\node at (14.5,0){7};\node at (16.5,0){1};\node at (17.5,0){2};\node at (18.5,0){3};\node at (19.5,0){4};\node at (20.5,0){5};\node at (21.5,0){6};\node at (22.5,0){7};

\node at (11.5,-1){\scriptsize$\beta^{4_1}=11$};\node at (11.5,-2){\scriptsize$\beta^{4}=11$};
\node at (19.5,-1){\scriptsize$\beta^{5_1}=111$};\node at (19.5,-2){\scriptsize$\beta^{5}=211$};
\end{tikzpicture}$

$\scriptsize\begin{tikzpicture}[scale=0.33]
\draw[line width=1pt] (0,0.5) --(7,0.5); \draw[line width=1pt] (8,0.5) --(15,0.5);
\draw[line width=1pt] (2,0.5) rectangle (3,1.5);\node at (2.5,1){3};
\draw[line width=1pt] (2,1.5) rectangle (3,2.5);\node at (2.5,2){3};
\draw[line width=1pt] (3,0.5) rectangle (4,1.5);\node at (3.5,1){4};
\draw[line width=1pt] (6,0.5) rectangle (7,1.5);\node at (6.5,1){7};
\draw[line width=1pt] (10,0.5) rectangle (11,1.5); \node at (10.5,1){3};
\draw[line width=1pt] (10,2.5) rectangle (11,3.5); \node at (10.5,3){2};
\draw[line width=1pt] (10,1.5) rectangle (11,2.5); \node at (10.5,2){3};
\draw[line width=1pt] (11,0.5) rectangle (12,1.5); \node at (11.5,1){4};
\draw[line width=1pt] (14,0.5) rectangle (15,1.5); \node at (14.5,1){7};
\node at (0.5,0){1};\node at (1.5,0){2};\node at (2.5,0){3};\node at (3.5,0){4};\node at (4.5,0){5};\node at (5.5,0){6};\node at (6.5,0){7};
\node at (8.5,0){1};\node at (9.5,0){2};\node at (10.5,0){3};\node at (11.5,0){4};\node at (12.5,0){5};\node at (13.5,0){6};\node at (14.5,0){7};
\node at (3.5,-1){\scriptsize$\beta^{5_2}=211$};\node at (3.5,-2){\scriptsize$\beta^{5}=211$};
\node at (11.5,-1){\scriptsize$\beta^{6_1}=311$};\node at (11.5,-2){\scriptsize$\beta^{6}=411$};
\end{tikzpicture}$
$~~~~\scriptsize\begin{tikzpicture}[scale=0.32]
\draw[line width=1pt] (16,0.5) --(23,0.5);



\draw[line width=1pt] (18,3.5) rectangle(19,4.5);\node at (18.5,4){2};
\draw[line width=1pt] (18,2.5) rectangle(19,3.5);\node at (18.5,3){2};
\draw[line width=1pt] (18,1.5) rectangle(19,2.5);\node at (18.5,2){3};
\draw[line width=1pt] (18,0.5) rectangle(19,1.5);\node at (18.5,1){3};
\draw[line width=1pt] (19,0.5) rectangle(20,1.5);\node at (19.5,1){4};
\draw[line width=1pt] (22,0.5) rectangle(23,1.5);\node at (22.5,1){7};
\node at (16.5,0){1};\node at (17.5,0){2};\node at (18.5,0){3};\node at (19.5,0){4};\node at (20.5,0){5};\node at (21.5,0){6};\node at (22.5,0){7};

\node at (19.5,-1){\scriptsize$\beta^{6_2}=411$};\node at (19.5,-2){\scriptsize$\beta^{6}=411$};
\end{tikzpicture}$$~~~\scriptsize\begin{tikzpicture}[scale=0.33]
\draw[line width=1pt] (0,0.5) --(7,0.5); \draw[line width=1pt] (8,0.5) --(15,0.5);
 \draw[line width=1pt] (16,0.5) --(23,0.5);
\draw[line width=1pt] (0,0.5) rectangle (1,1.5);\node at (0.5,1){1};
\draw[line width=1pt] (2,3.5) rectangle (3,4.5);\node at (2.5,4){2};
\draw[line width=1pt] (2,2.5) rectangle (3,3.5);\node at (2.5,3){2};
\draw[line width=1pt] (2,0.5) rectangle (3,1.5);\node at (2.5,1){3};
\draw[line width=1pt] (2,1.5) rectangle (3,2.5);\node at (2.5,2){3};
\draw[line width=1pt] (3,0.5) rectangle (4,1.5);\node at (3.5,1){4};
\draw[line width=1pt] (6,0.5) rectangle (7,1.5);\node at (6.5,1){7};

\draw[line width=1pt] (8,1.5) rectangle (9,2.5); \node at (8.5,2){1};
\draw[line width=1pt] (8,0.5) rectangle (9,1.5); \node at (8.5,1){1};
\draw[line width=1pt] (10,3.5) rectangle (11,4.5); \node at (10.5,4){2};
\draw[line width=1pt] (10,0.5) rectangle (11,1.5); \node at (10.5,1){3};
\draw[line width=1pt] (10,2.5) rectangle (11,3.5); \node at (10.5,3){2};
\draw[line width=1pt] (10,1.5) rectangle (11,2.5); \node at (10.5,2){3};
\draw[line width=1pt] (11,0.5) rectangle (12,1.5); \node at (11.5,1){4};
\draw[line width=1pt] (14,0.5) rectangle (15,1.5); \node at (14.5,1){7};

\draw[line width=1pt] (16,2.5) rectangle(17,3.5);\node at (16.5,3){1};
\draw[line width=1pt] (16,1.5) rectangle(17,2.5);\node at (16.5,2){1};
\draw[line width=1pt] (16,0.5) rectangle(17,1.5);\node at (16.5,1){1};
\draw[line width=1pt] (18,3.5) rectangle(19,4.5);\node at (18.5,4){2};
\draw[line width=1pt] (18,2.5) rectangle(19,3.5);\node at (18.5,3){2};
\draw[line width=1pt] (18,1.5) rectangle(19,2.5);\node at (18.5,2){3};
\draw[line width=1pt] (18,0.5) rectangle(19,1.5);\node at (18.5,1){3};
\draw[line width=1pt] (19,0.5) rectangle(20,1.5);\node at (19.5,1){4};
\draw[line width=1pt] (22,0.5) rectangle(23,1.5);\node at (22.5,1){7};
\node at (0.5,0){1};\node at (1.5,0){2};\node at (2.5,0){3};\node at (3.5,0){4};\node at (4.5,0){5};\node at (5.5,0){6};\node at (6.5,0){7};\node at (8.5,0){1};\node at (9.5,0){2};\node at (10.5,0){3};\node at (11.5,0){4};\node at (12.5,0){5};\node at (13.5,0){6};\node at (14.5,0){7};\node at (16.5,0){1};\node at (17.5,0){2};\node at (18.5,0){3};\node at (19.5,0){4};\node at (20.5,0){5};\node at (21.5,0){6};\node at (22.5,0){7}; \node at (24,0.45){=F};

\node at (3.5,-1){\scriptsize$\beta^{7_1}=4111$};\node at (3.5,-2){\scriptsize$\beta^{7}=4311$};
\node at (11.5,-1){\scriptsize$\beta^{7_2}=4211$};\node at (11.5,-2){\scriptsize$\beta^{7}=4311$};
\node at (19.5,-1){\scriptsize$\beta^{7_3}=4311$};\node at (19.5,-2){\scriptsize$\beta^{7}=4311$};
\end{tikzpicture}$
\section{The Bruhat order in $\mathfrak{S}_n$ and orbits.}
\label{sec:bruhat}
 Let $\theta=\theta_1\ldots \theta_n\in\mathfrak{S}_n$, written in one  line notation. A pair $(i,j)$, with $i<j$, such that $\theta_i>\theta_j$, is said to be an inversion of $\theta$, and $\ell(\theta)$ denotes the number of inversions of $\theta$.
The Bruhat order in $\mathfrak{S}_n$ is  the partial order in  $\mathfrak{S}_n$ defined by the transitive closure of the relations:
$\label{closure}\nonumber\theta<t\theta$, if $\ell(\theta)<\ell(t\theta)$, with $t$ a transposition, $\theta\in \mathfrak{S}_n$.
Let $\theta=s_{i_N}\cdots s_{i_1}$ be a decomposition of $\theta$ into simple transpositions $s_i=(i\;i+1)$, $1\le i<n$. When $N=\ell(\theta)$, the number $N$ in a such decomposition is minimised, and it is said to be a  reduced decomposition of $\theta$. The longest permutation of $\mathfrak{S}_n$ is denoted by $\omega$.
Let $\lambda$ be a partition in $\mathbb{N}^n$. The Bruhat  ordering of the orbit of $\lambda$, $\mathfrak{S}_n\lambda$, is defined by taking the transitive closure
 of the relations
 $\label{inducedbruhat}\nonumber\alpha< t\alpha$, if $\alpha_i>\alpha_{j}$, $i<j$,  and  $t$ the transposition $(i\,j)$, $\alpha\in \mathfrak{S}_n\lambda$.
Given $\alpha\in \mathbb{N}^n$, a pair  $(i,j)$, with $i<j$, such that $\alpha_i<\alpha_j$, is called  an inversion of $\alpha$, and   $\iota(\alpha)$  denotes the number of inversions of $\alpha$.
 We may write $\alpha<\beta$ if $\iota(\alpha)<\iota(\beta)$ and $\beta=\tau\alpha$ for some permutation $\tau$ in $\mathfrak{S}_n$ that can be written as a product of transpositions each increasing the number of inversions when passing from $\alpha$ to $\beta$. Recalling the exchange condition and the property of the Bruhat ordering   which says that if $\theta\le \sigma$ and $s$ is a simple transposition then either $s\theta\le \sigma$ or $s\theta\le s\sigma$, \cite{humphreys}, one has the useful.
\begin{lem}\label{1layer-easy}
Let  $\alpha$, $\beta\in \mathbb{N}^n$ be rearrangements of each other.
Let $1\le k<n$ and $1\le$ $ r_1<$ $r_2<$ $\dots$$<$ $r_k<n$. Then

$(a)$ $\alpha>s_{r_1}\alpha$ 
$\Rightarrow $
$ s_{{r_k}}\dots s_{{r_2}}\alpha>s_{{r_k}}\dots s_{{r_2}}s_{{r_1}}\alpha$ and  $~~\alpha<s_{r_1}\alpha$ 
$ \Rightarrow$
$ s_{{r_k}}\dots s_{{r_2}}\alpha<s_{{r_k}}\dots s_{{r_2}}s_{{r_1}}\alpha$.

$(b)$  $\alpha>s_{r_1}\alpha$ and 
$s_{{r_{k-1}}}\dots s_{{r_2}}\alpha$ $>$ $s_{r_{k}} s_{r_{k-1}}$ $\dots $ $s_{{r_2}}$ $\alpha$   $ \Rightarrow$ $s_{{r_{k-1}}}\dots s_{{r_2}}s_{{r_1}}\alpha> s_{r_k}s_{{r_{k-1}}}\dots s_{{r_2}} s_{{r_1}}\alpha.$

$(c)$
 $\beta\nleq  \omega s_{r_k}\cdots \widehat s_{r_{i}}\cdots s_{r_1}\alpha$, $1\le i<k$, and
 $\beta $ $\leq $ $\omega $ $s_{r_k}$ $\cdots$  $s_{r_1}$ $\alpha$ $ \Rightarrow$
$s_{r_{k}}\cdots  s_{r_{1}}\alpha<\dots < s_{r_{1}}\alpha<\alpha$.

$(d)$ \label{nice} 
 $\beta\le s_k\alpha$ and $\beta\nleq
\alpha$ iff  $s_k\beta\leq \alpha$  and $\beta\nleq
\alpha$.
 Moreover, $\alpha_k>\alpha_{k+1}$ and $\beta_k<\beta_{k+1}$.
\end{lem}
\section{Creation and annihilation of corner cells in a Ferrers shape. 
}
\label{sec:crys}
Crystal operators or coplactic operations $e_r$, $f_r$, $1\le r<n$, are defined on any word over the alphabet $[n]$, see \cite{thibon} and \cite{ hoon}. These operations can  be extended to biwords  in two ways. Either by considering $w$ in lexicographic order, with respect to the first row, or to the second. Let  $w=\binom{\bf u}
{\bf v}$ be in lexicographic order, with respect to the first row. Write $e_rw:=\binom{\bf u}
{e_r\bf v}$ and,  similarly, for
  $f_rw$.
 Let $\binom{\bf k}
{\bf l}$ be  the  biword $w$ rearranged in lexicographic order, with respect to the second row.
  Write $\overline{w}:=\binom{\bf v}
{\bf u}$ and $ e^*_rw:=e_r\overline{w}=\binom{\bf l}
{e_r\bf k}$ and,  similarly, for
  ${f}_r^*w$.
  The resulting biwords are still in lexicographic order with respect to the first row. A biword $w$ can be seen as a multiset of cells in a Ferrers shape $\lambda$,
  embedded in a $n  \times n$ square by putting  a cross "${\bf X}$" in the cell $(i,j)$ of $\lambda$ for each biletter $\binom{j}{i}$ in $w$.
As our running example, consider the  biword 
$\label{biw}\scriptsize
 w=\left(\begin{array}{cccccccccccccccccc}
 1&1&2&2&3&3&3&4&4&4&5&5&5&6&7\\4&4&1&3&5&5&5&3&3&4&3&4&4&2&2\end{array}\right)$
 in lexicographic order, with respect to the first row,  over the alphabet $[7]$, and its representation in the Ferrers shape
$\begin{tikzpicture}[scale=0.42]
 \draw[line width=1pt] (6,1) rectangle (7,2);
\draw[line width=1pt] (0,0) rectangle (7,1);
\draw[line width=1pt] (0,1) rectangle (6,2);\draw[line width=1pt] (0,2) rectangle (5,3);\draw[line width=1pt] (0,3) rectangle (5,4);\draw[line width=1pt] (0,4) rectangle (3,5);\draw[line width=0.95pt] (0,5) rectangle (2,5);\draw[line width=0.95pt] (0,4) rectangle (1,4);
\draw[line width=1pt] (1,0) -- (1,5);\draw[line width=1pt] (2,0) -- (2,5);\draw[line width=1pt] (3,0) -- (3,5);\draw[line width=0.95pt] (4,0) -- (4,4);\draw[line width=1pt] (5,0) --(5,4);\draw[line width=1pt] (6,0) -- (6,2);
\node[color=brown] at (0.3,3.3){\small$\bf X$}; \node[color=blue] at (0.7,3.7){\small$\bf X$}; \node at (1.5,0.5){$\bf X$}; \node[color=blue]  at (1.5,2.5){$\bf X$};\node  at (2.25,4.3){\tiny$\bf X$}; \node at (2.5,4.5){{\tiny$\bf X$}};
 \node at (2.75,4.7){{\tiny$\bf X$}};
\node[color=brown]  at (3.3,2.3){{\small$\bf X$}};
\node  at (3.7,2.7){{\small$\bf X$}};\node[color=red]  at (3.5,3.5){{\small$\bf X$}};\node[color=red]  at (4.5,2.5){{$\bf X$}};\node  at (4.3,3.3){\small$\bf X$};\node  at (4.7,3.6){\small$\bf X$};
\node  at (5.5,1.5){$\bf X$};  \node at (6.5,1.5){$\bf X$};
\end{tikzpicture}$.
The crystal operator $e_3$  acts on $w$ through its action on the second row of $w$. 
 Ignore  all  entries different from $3$ and $4$, 
 obtaining the subword $443334344$. Match, in the usual way, all $43$ (in blue in the example below), remaining the subword $344$. Change to $3$ the leftmost $4$, giving  $334$.   Applying again the operator $e_3$
 means to apply 
 $e_3$ to the subword $344$ which change to $333$,  obtaining
 $\scriptsize\left(\begin{array}{cccccccccccccccc}
1&1&2&4&4&4&5&5&5\\{\color{brown}\bold 4}&{\color{blue}\bold 4}&{\color{blue}\bold 3}&{\color{brown}\bold 3}&\bold 3&{\color{red}\bold 4}&{\color{red}\bold 3}
&{\bold 4}&{\bold 4}
\end{array}\right)\overset{e^2_3}{\underset{f^2_3}{\rightleftarrows}}
\left(\begin{array}{cccccccccccccccc}
1&1&2&4&4&4&5&5&5\\{\color{brown}\bold 4}&{\color{blue}\bold 4}&{\color{blue}\bold 3}&{\color{brown}\bold 3}&{\bold 3}&{\color{red}\bold 4}&{\color{red}\bold 3}
&{\bold 3}&{\bold 3}
\end{array}\right)
$.
    The action of  $e_r$ or $f_r$, as long as it is possible, on the second row of  $w$, translates to the rows $r$ and $r+1$ of the Ferrers shape as a matching of crosses followed by sliding the unmatched crosses in those rows. Write $\Upsilon_r$ ($\overline\Upsilon_r$) for applying $m$ times the crystal operator $e_r$ ($f_r$), to the second row of $w$, where $m$ is the number of unmatched $r+1$ ($r$) in the second row of $w$, then
     $\scriptsize\begin{tikzpicture}[scale=0.3]
\draw[line width=1pt] (0,0) rectangle (8,2);
\draw[line width=1pt] (0,1) -- (8,1);
\draw[line width=0.3pt] (1,1) -- (1,2);
\draw[line width=1pt] (2,0) -- (2,2);
\draw[line width=1pt] (3,0) -- (3,2);
\draw[line width=1pt] (4,0) -- (4,2);
\draw[line width=0.3pt] (5,0) -- (5,1);
\draw[line width=1pt] (6,0) -- (6,2);
\draw[line width=0.3pt] (7,1) -- (7,2);
\node[color=brown] at (0.5,1.5){ \bf{X}};\node[color=blue] at (1.5,1.5){ \bf{X}};
\node[color=blue] at (2.5,0.5){ \bf{X}};
\node[color=brown] at (4.5,0.5){ \bf{X}};\node at (5.5,0.5){ \bf{X}};
\node[color=red] at (5,1.5){ \bf{X}};
\node at (6.5,1.5){ \bf{X}};\node at (7.5,1.5){ \bf{X}};
\node[color=red] at (7,0.5){ \bf{X}};
\end{tikzpicture}$ $\scriptsize\begin{tikzpicture}[scale=0.3]
 \draw[line width=1pt] (0,0) rectangle (9,2);
\draw[line width=1pt] (0,1) -- (9,1);
\draw[line width=0.3pt] (1,1) -- (1,2);
\draw[line width=1pt] (2,0) -- (2,2);
\draw[line width=1pt] (3,0) -- (3,2);
\draw[line width=1pt] (4,0) -- (4,2);
\draw[line width=0.3pt] (5,0) -- (5,1);
\draw[line width=1pt] (6,0) -- (6,2);
\draw[line width=0.3pt] (7,0) -- (7,1);
\draw[line width=0.3pt] (8,0) -- (8,1);
\node[color=brown] at (0.5,1.5){ \bf{X}};\node[color=blue] at (1.5,1.5){ \bf{X}};
\node[color=blue] at (2.5,0.5){ \bf{X}};
\node[color=brown] at (4.5,0.5){ \bf{X}};\node at (5.5,0.5){ \bf{X}};
\node[color=red] at (5,1.5){ \bf{X}};
\node at (7.5,0.5){ \bf{X}};\node at (8.5,0.5){ \bf{X}};
\node[color=red] at (6.5,0.5){ \bf{X}};
\node at (-2,1.3){$\longrightarrow$};
\node at (-2,1.8){$\tiny \Upsilon_3$};\node at (-2,0.0){$\longleftarrow$};\node at (-2,0.6){$\tiny \overline \Upsilon_3$};
\end{tikzpicture}
$.
 Consider now the $01$-filling representation of the  biwords $w$  and $\Upsilon_r{w}$  in the Ferrers shape $\lambda$ embedded in a rectangle shape.
 Apply the local rules, as defined in  Section \ref{sec:rsk}. Notice that in the $01$-filling of ${w}$, we match a cross in row $r+1$ with a cross to the SE, in row $r$, such that in these two rows  there is no unmatched cross in a column between them.
These two  growth diagrams have the same bottom sequences of partitions  and the left sequences   differ only in the partitions assigned to the  rows $r$ and $r+1.$
Let $w_r$ and $\widetilde{w}_r$ be the  biwords  obtained from  $w$ and $\Upsilon_r{w},$ after deleting all the biletters  with bottom rows different from $r$ and $r+1$.
The translation of the movement of the cells  in the  Ferrers shape to the 01-filling is as follows.
In the $01$-filling of $w_r$ move up, without changing of columns, the matched crosses of row $r+1$, say $s$ crosses, to the top most $s$ rows such that they form SW chain. Then slide down  the remaining unmatched  crosses, from row $r+1$ to row $r$, without changing of columns, such that these crosses and all the crosses of row $r$ form a SW chain.  The result is the 01-filling corresponding to $\widetilde{w}_r.$  We illustrate with our running example.

$
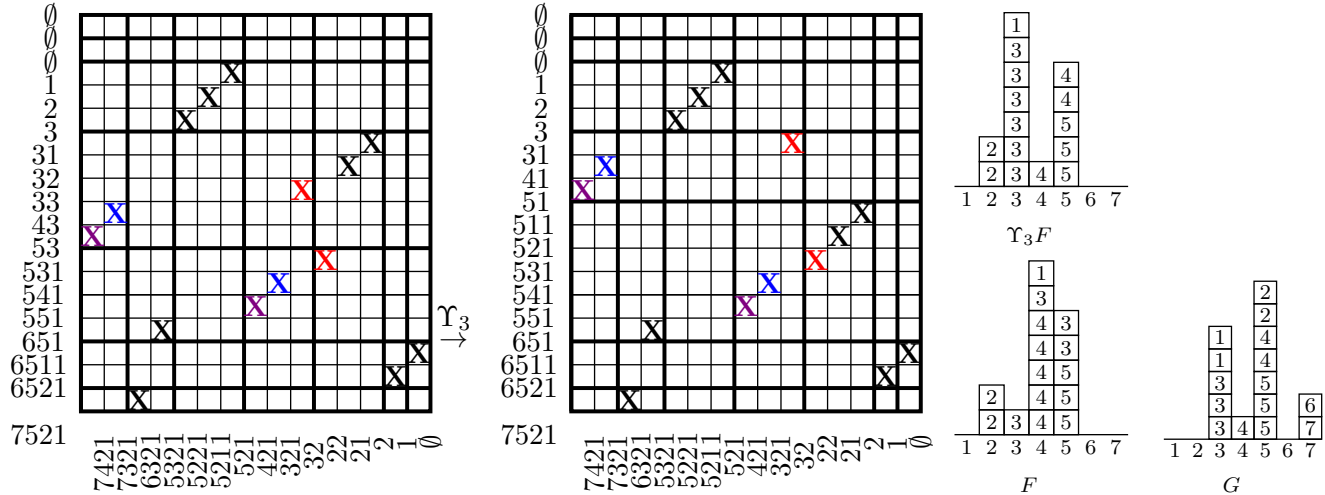
\begin{tikzpicture}[scale=0.31]
\draw[line width=1.5pt] (0,0) rectangle (15,17);
\draw[line width=0.5pt] (0,0) -- (15,0);
\draw[line width=1.5pt] (0,1) -- (15,1);
\draw[line width=0.5pt] (0,2) -- (15,2);
\draw[line width=1.5pt] (0,3) -- (15,3);
\draw[line width=0.5pt] (0,4) -- (15,4);
\draw[line width=0.5pt] (0,5) -- (15,5);
\draw[line width=0.5pt] (0,5) -- (15,5);
\draw[line width=0.5pt] (0,6) -- (15,6);
\draw[line width=1.5pt] (0,7) -- (15,7);
\draw[line width=0.5pt] (0,8) -- (15,8);
\draw[line width=0.5pt] (0,9) -- (15,9);
\draw[line width=0.5pt] (0,10) -- (15,10);
\draw[line width=0.5pt] (0,11) -- (15,11);
\draw[line width=1.5pt] (0,12) -- (15,12);
\draw[line width=0.5pt] (0,13) -- (15,13);
\draw[line width=0.5pt] (0,14) -- (15,14);
\draw[line width=1.5pt] (0,15) -- (15,15);
\draw[line width=1.5pt] (0,16) -- (15,16);
\draw[line width=1.5pt] (0,0) -- (0,17);
\draw[line width=0.5pt] (1,0) -- (1,17);
\draw[line width=1.5pt] (2,0) -- (2,17);
\draw[line width=0.5pt] (3,0) -- (3,17);
\draw[line width=1.5pt] (4,0) -- (4,17);
\draw[line width=0.5pt] (5,0) -- (5,17);
\draw[line width=0.5pt] (6,0) -- (6,17);
\draw[line width=1.5pt] (7,0) -- (7,17);
\draw[line width=0.5pt] (8,0) -- (8,17);
\draw[line width=0.5pt] (9,0) -- (9,17);
\draw[line width=1.5pt] (10,0) -- (10,17);
\draw[line width=0.5pt] (11,0) -- (11,17);
\draw[line width=0.5pt] (12,0) -- (12,17);
\draw[line width=1.5pt] (13,0) -- (13,17);
\draw[line width=1.5pt] (14,0) -- (14,17);
\node[color=violet]  at (0.5,7.5){$\bf X$};\node[color=blue] at
(1.5,8.5){$\bf X$};\node at (2.5,0.5){$\bf X$};\node at
(3.5,3.5){$\bf X$};\node  at (4.5,12.5){$\bf X$};\node at
(5.5,13.5){$\bf X$};\node at (6.5,14.5){$\bf X$};
\node[color=violet]  at (7.5,4.5){$\bf X$};\node[color=blue]  at
(8.5,5.5){$\bf X$};\node[color=red] at
(9.5,9.5){$\bf X$};\node[color=red]  at (10.5,6.5){$\bf X$};
\node at (11.5,10.5){$\bf X$};\node at (12.5,11.5){$\bf X$};\node at
(13.5,1.5){$\bf X$};\node at (14.5,2.5){$\bf X$};
\node at (8,18){}; \node at (16,3){$\rightarrow$};\node at
(16,4){$\Upsilon_3$};
\node at(-1.3,17){$\emptyset$};\node at(-1.3,16){$\emptyset$};\node
at(-1.3,15){$\emptyset$};\node at(-1.3,14){1};
\node at(-1.3,13){2};\node at(-1.3,12){3};\node at
(-1.5,11){31};\node at (-1.5,10){32};\node at (-1.5,9){33};\node at
(-1.5,8){43};\node at
(-1.5,7){53};\node at (-1.6,6){531};\node at (-1.6,5){541};\node at
(-1.6,4){551};\node at (-1.6,3){651};\node at (-1.8,2){6511};\node at
(-1.8,1){6521};\node at (-1.8,-1){7521};
 \node at (15,-1.3)  [rotate=90] {$\emptyset$};  \node at (14,-1.4)
[rotate=90] {{\small1}}; \node at (13,-1.4)  [rotate=90] {{\small2}};
\node at (12,-1.6)  [rotate=90] {{\small21}}; \node at (11,-1.6)
[rotate=90] {{\small22}};\node at (10,-1.9) [ rotate=90]{{\small32}};
\node at (9,-1.9) [ rotate=90]{{\small321}};\node at (8,-1.9)
[ rotate=90]{{\small421}};\node at (7,-1.9) [
rotate=90]{{\small521}};\node at (6,-2.1) [
rotate=90]{{\small5211}};\node at (5,-2.1) [
rotate=90]{{\small5221}};\node at
(4,-2.1) [ rotate=90]{{\small5321}};\node at (3,-2.1) [
rotate=90]{{\small6321}};\node at (2,-2.2) [rotate=90]{{\small7321}};\node
at (1,-2.2) [rotate=90]{{\small7421}};
\end{tikzpicture}$
$\begin{tikzpicture}[scale=0.31]
\draw[line width=1.5pt] (0,0) rectangle (15,17);
\draw[line width=0.5pt] (0,0) -- (15,0);
\draw[line width=1.5pt] (0,1) -- (15,1);
\draw[line width=0.5pt] (0,2) -- (15,2);
\draw[line width=1.5pt] (0,3) -- (15,3);
\draw[line width=0.5pt] (0,4) -- (15,4);
\draw[line width=0.5pt] (0,5) -- (15,5);
\draw[line width=0.5pt] (0,5) -- (15,5);
\draw[line width=0.5pt] (0,6) -- (15,6);
\draw[line width=0.5pt] (0,7) -- (15,7);
\draw[line width=0.5pt] (0,8) -- (15,8);
\draw[line width=1.5pt] (0,9) -- (15,9);
\draw[line width=0.5pt] (0,10) -- (15,10);
\draw[line width=0.5pt] (0,11) -- (15,11);
\draw[line width=1.5pt] (0,12) -- (15,12);
\draw[line width=0.5pt] (0,13) -- (15,13);
\draw[line width=0.5pt] (0,14) -- (15,14);
\draw[line width=1.5pt] (0,15) -- (15,15);
\draw[line width=1.5pt] (0,16) -- (15,16);
\draw[line width=1.5pt] (0,0) -- (0,17);
\draw[line width=0.5pt] (1,0) -- (1,17);
\draw[line width=1.5pt] (2,0) -- (2,17);
\draw[line width=0.5pt] (3,0) -- (3,17);
\draw[line width=1.5pt] (4,0) -- (4,17);
\draw[line width=0.5pt] (5,0) -- (5,17);
\draw[line width=0.5pt] (6,0) -- (6,17);
\draw[line width=1.5pt] (7,0) -- (7,17);
\draw[line width=0.5pt] (8,0) -- (8,17);
\draw[line width=0.5pt] (9,0) -- (9,17);
\draw[line width=1.5pt] (10,0) -- (10,17);
\draw[line width=0.5pt] (11,0) -- (11,17);
\draw[line width=0.5pt] (12,0) -- (12,17);
\draw[line width=1.5pt] (13,0) -- (13,17);
\draw[line width=1.5pt] (14,0) -- (14,17);
\node[color=violet]  at (0.5,9.5){$\bf X$};\node[color=blue] at
(1.5,10.5){$\bf X$};\node at (2.5,0.5){$\bf X$};\node at
(3.5,3.5){$\bf X$};\node  at (4.5,12.5){$\bf X$};\node at
(5.5,13.5){$\bf X$};\node at (6.5,14.5){$\bf X$};
\node[color=violet]  at (7.5,4.5){$\bf X$};\node[color=blue]  at
(8.5,5.5){$\bf X$};\node[color=red] at
(9.5,11.5){$\bf X$};\node[color=red]  at (10.5,6.5){$\bf X$};
\node at (11.5,7.5){$\bf X$};\node at (12.5,8.5){$\bf X$};\node at
(13.5,1.5){$\bf X$};\node at (14.5,2.5){$\bf X$};
\node at (8,18){};
\node at(-1.3,17){$\emptyset$};\node at(-1.3,16){$\emptyset$};\node
at(-1.3,15){$\emptyset$};\node at(-1.3,14){1};
\node at(-1.3,13){2};\node at(-1.3,12){3};\node at
(-1.5,11){31};\node at (-1.5,10){41};\node at (-1.5,9){51};\node at
(-1.6,8){511};\node at
(-1.6,7){521};\node at (-1.6,6){531};\node at (-1.6,5){541};\node at
(-1.6,4){551};\node at (-1.6,3){651};\node at (-1.8,2){6511};\node at
(-1.8,1){6521};\node at (-1.8,-1){7521};
 \node at (15,-1.3)  [rotate=90] {$\emptyset$};  \node at (14,-1.4)
[rotate=90] {{\small1}}; \node at (13,-1.4)  [rotate=90] {{\small2}};
\node at (12,-1.6)  [rotate=90] {{\small21}}; \node at (11,-1.6)
[rotate=90] {{\small22}};\node at (10,-1.9) [ rotate=90]{{\small32}};
\node at (9,-1.9) [ rotate=90]{{\small321}};\node at (8,-1.9)
[ rotate=90]{{\small421}};\node at (7,-1.9) [
rotate=90]{{\small521}};\node at (6,-2.1) [
rotate=90]{{\small5211}};\node at (5,-2.1) [
rotate=90]{{\small5221}};\node at
(4,-2.1) [ rotate=90]{{\small5321}};\node at (3,-2.1) [
rotate=90]{{\small6321}};\node at (2,-2.2) [rotate=90]{{\small7321}};\node
at (1,-2.2) [rotate=90]{{\small7421}};
\end{tikzpicture}$
$\begin{tikzpicture}[scale=0.33]
\draw[line width=0.5pt] (12,3) -- (19,3);
\node at (12.5,2.5){\scriptsize 1};\node at (13.5,2.5){\scriptsize 2};\node at (14.5,2.5){\scriptsize 3};\node
at (15.5,2.5){\scriptsize 4};\node at (16.5,2.5){\scriptsize 5};\node at (17.5,2.5){\scriptsize 6};\node at
(18.5,2.5){\scriptsize 7};
\draw[line width=0.5pt] (13,3) rectangle  (14,4);\draw[line width=0.5pt]
(13,4) rectangle  (14,5);
\draw[line width=0.5pt] (14,3) rectangle  (15,4);\draw[line width=0.5pt]
(15,4) rectangle  (16,5);\draw[line width=0.5pt] (15,5) rectangle
(16,6);\draw[line width=0.5pt] (15,6) rectangle  (16,7);\draw[line
width=0.5pt] (15,7) rectangle  (16,8);\draw[line width=0.5pt] (15,8)
rectangle  (16,9);\draw[line width=0.5pt] (15,9) rectangle  (16,10);
\draw[line width=0.5pt] (15,3) rectangle  (16,4);\draw[line width=0.5pt]
(16,6) rectangle  (17,7);\draw[line width=0.5pt] (16,7) rectangle  (17,8);
\draw[line width=0.5pt] (16,3) rectangle  (17,4);\draw[line width=0.5pt]
(16,4) rectangle  (17,5);\draw[line width=0.5pt] (16,5) rectangle  (17,6);
\node at (13.5,3.5){\scriptsize 2};\node at (13.5,4.5){\scriptsize 2};
\node at (14.5,3.5){\scriptsize 3};\node at (15.5,4.5){\scriptsize 4};\node at (15.5,5.5){\scriptsize 4};\node
at (15.5,6.5){\scriptsize 4};\node at (15.5,7.5){\scriptsize 4};\node at (15.5,8.5){\scriptsize 3};\node at
(15.5,9.5){\scriptsize 1};
\node at (15.5,3.5){\scriptsize 4};\node at (16.5,7.5){\scriptsize 3};\node at (16.5,6.5){\scriptsize 3};
\node at (16.5,3.5){\scriptsize 5};\node at (16.5,4.5){\scriptsize 5};\node at (16.5,5.5){\scriptsize 5};
\node at (15,1){\scriptsize$F$};

\draw[line width=0.5pt] (12,13) -- (19,13);
\node at (12.5,12.5){\scriptsize 1};
\node at (13.5,12.5){\scriptsize 2};
\node at (14.5,12.5){\scriptsize 3};
\node at (15.5,12.5){\scriptsize 4};
\node at (16.5,12.5){\scriptsize 5};
\node at (17.5,12.5){\scriptsize 6};
\node at (18.5,12.5){\scriptsize 7};
\draw[line width=0.5pt] (13,13) rectangle  (14,14);\draw[line width=0.5pt]
(13,14) rectangle  (14,15);
\draw[line width=0.5pt] (14,13) rectangle  (15,14);\draw[line width=0.5pt]
(14,14) rectangle  (15,15);\draw[line width=0.5pt] (14,15) rectangle
(15,16);\draw[line width=0.5pt] (14,16) rectangle  (15,17);\draw[line
width=0.5pt] (14,17) rectangle  (15,18);\draw[line width=0.5pt] (14,18)
rectangle  (15,19);\draw[line width=0.5pt] (14,19) rectangle  (15,20);
\draw[line width=0.5pt] (15,13) rectangle  (16,14);\draw[line width=0.5pt]
(16,16) rectangle  (17,17);\draw[line width=0.5pt] (16,17) rectangle  (17,18);
\draw[line width=0.5pt] (16,13) rectangle  (17,14);\draw[line width=0.5pt]
(16,14) rectangle  (17,15);\draw[line width=0.5pt] (16,15) rectangle  (17,16);
\node at (13.5,13.5){\scriptsize 2};
\node at (13.5,14.5){\scriptsize 2};
\node at (14.5,13.5){\scriptsize 3};
\node at (14.5,14.5){\scriptsize 3};
\node at (14.5,15.5){\scriptsize 3};
\node at (14.5,16.5){\scriptsize 3};
\node at (14.5,17.5){\scriptsize 3};
\node at (14.5,18.5){\scriptsize 3};
\node at (14.5,19.5){\scriptsize 1};
\node at (15.5,13.5){\scriptsize 4};
\node at (16.5,17.5){\scriptsize 4};
\node at (16.5,16.5){\scriptsize 4};
\node at (16.5,13.5){\scriptsize 5};
\node at (16.5,14.5){\scriptsize 5};
\node at (16.5,15.5){\scriptsize 5};
\node at (15,11){\scriptsize $\Upsilon_3{F}$};

\end{tikzpicture}$
$~~\begin{tikzpicture}[scale=0.3]
\draw[line width=0.5pt] (20,3) -- (27,3);
\node at (20.5,2.5){\scriptsize 1};\node at (21.5,2.5){\scriptsize 2};\node at (22.5,2.5){\scriptsize 3};\node
at (23.5,2.5){\scriptsize 4};\node at (24.5,2.5){\scriptsize 5};\node at (25.5,2.5){\scriptsize 6};\node at
(26.5,2.5){\scriptsize 7};
\draw[line width=0.5pt] (22,3) rectangle  (23,4);
\draw[line width=0.5pt] (22,4) rectangle  (23,5);
\draw[line width=0.5pt] (22,5) rectangle  (23,6);\draw[line width=0.5pt]
(22,6) rectangle  (23,7);\draw[line width=0.5pt] (22,7) rectangle  (23,8);
\draw[line width=0.5pt] (23,3) rectangle  (24,4);
\draw[line width=0.5pt] (24,3) rectangle  (25,4);\draw[line width=0.5pt]
(24,4) rectangle  (25,5);\draw[line width=0.5pt] (24,5) rectangle  (25,6);
\draw[line width=0.5pt] (24,6) rectangle  (25,7);\draw[line width=0.5pt]
(24,7) rectangle  (25,8);
\draw[line width=0.5pt] (24,8) rectangle  (25,9);\draw[line width=0.5pt]
(24,9) rectangle  (25,10);
\draw[line width=0.5pt] (26,3) rectangle  (27,4);
\draw[line width=0.5pt] (26,4) rectangle  (27,5);
\node at (22.5,3.5){\scriptsize 3};\node at (22.5,4.5){\scriptsize 3};
\node at (22.5,5.5){\scriptsize 3};\node at (22.5,6.5){\scriptsize 1};\node at (22.5,7.5){\scriptsize 1};
\node at (23.5,3.5){\scriptsize 4};
\node at (24.5,3.5){\scriptsize 5};\node at (24.5,4.5){\scriptsize 5};\node at (24.5,5.5){\scriptsize 5};
\node at (24.5,6.5){\scriptsize 4};\node at (24.5,7.5){\scriptsize 4};\node at (24.5,8.5){\scriptsize 2};\node
at (24.5,9.5){\scriptsize 2};\node at (26.5,3.5){\scriptsize 7};\node at (26.5,4.5){\scriptsize 6};
\node at (23,1){\scriptsize $G$};
\end{tikzpicture}$
\begin{thm} \label{corcor1}  Let  $\lambda$ be a Ferrers shape and $w$ a biword  consisting of  a multiset of cells of $\lambda$ containing the cell $(r+1,\lambda_{r+1})$ with multiplicity at least one. Let $\Phi(w)=(F,G)$ where $sh(F)=\nu$ and $sh(G)=\beta$.
The following holds

$(a)$ If $\nu_r<\nu_{r+1}$, for some $r\ge 1$, then  $sh(\Upsilon_r F)=s_r\nu.$

$(b)$ $\Phi(\Upsilon_r w)=(\Upsilon_r{F},G)$ and  $\Phi(\Upsilon_r\overline w)=(\Upsilon_r{G,F})$.

$(c)$ If  $\lambda_r=\lambda_{r+1}> \lambda_{r+2}\ge 0$, for some $r\ge 1$, then $\nu_r<\nu_{r+1}$ and
$sh(\Upsilon_r{F})=s_r\nu$. Moreover,
$\Upsilon_r{w}$  fits the Ferrers shape $\lambda$ with the cell $(r+1,\lambda_{r+1})$ deleted.

$(d)$ If $\mu:=\overline\lambda$ is the transpose of the Ferrers shape $\lambda$ , and $\mu_{r}=\mu_{r+1}> \mu_{r+2}\ge 0$, for some $r\ge 1$,
  then $\beta_{r}<\beta_{{r+1}}$  and
$sh(\Upsilon_r{G})=s_{r}\beta$. Moreover, $\Upsilon_r{\overline w}$
fits the Ferrers shape $\lambda$ with the cell  $(\mu_{r+1},{r+1})$
deleted.
\end{thm}
\begin{proof} $(a)$ This is a consequence of the inductive definition of Demazure crystal  in the crystal structure of tableaux. \cite{masondemazure} has translated the operator $f_r$ on SSYT's into the operator $\Theta_r$ on SSAF's.We use it to show how the $sh(F)$ change under the action of $\Upsilon_r $. Since $\nu_r<\nu_{r+1}$, from Remark 1,
the  triples in columns
$r$ and $r + 1$ of $F$ are so that if  $r$ appears
$k$ times in the column $r$, there should be $r+1$  at least $k + 1$ times in the
column $r + 1$. Hence there is at least one unmatched $r + 1$ in column
$r + 1$.  The reverse of the operator $\Theta_r$,   the analogue of $e_r$, applied as long as  possible, gives
$sh(\Upsilon_r F)=s_r\nu$.

$(b)$ Consider  the $01$-filling of the  biwords $w$  and $\Upsilon_r{w}$  and their RRSK growth diagrams.  The operations $e_r$ and $f_r$ are coplactic, see \cite{thibon}. The two  growth diagrams have then the same bottom sequence of partitions and $\Phi(\Upsilon_r w)=(\Upsilon_r{F},G)$. Transposing the RRSK  growth diagram, through the NE-SW diagonal, gives the RRSK growth diagrams of $\overline w$ and  $\Upsilon_r\overline w$, respectively. Thus  $\Phi(\Upsilon_r\overline w)=(\Upsilon_r{G},F)$.

$(c)$ One  just has to prove that $\nu_r<\nu_{r+1}$. Let $\mathcal{G}$ be the growth diagram of the RRSK of $w$. Consider its left side. Since $\lambda_{r+1}>\lambda_{r+2}$, the partition, in the thin row below the thick row $r+1$ of  $\mathcal{G}$, has one more component. At this level, then the SSAF has  $\nu_{r+1}>0$. When we arrive at the thick row $r$, the SSAF has $\nu_{r+1}\ge 1$ and $\nu_{i}=0$, for $i\le r$. Since $\lambda_r=\lambda_{r+1}$,  $\mathcal{G}$ has at least one cross, say $\bf X$, contributing with $r+1$, to the right of the rightmost cross between thick rows $r$ and $r-1$.
 Either a new component is created in the partition, immediately below thick row $r$, or not. In the last case, when we arrive at thick row $r-1$, we have $\nu_r=0$, and   the crosses below  will not sit a box on column $r$ of the SSAF, thus still $\nu_r=0$. In the former case, $\nu_r\ge 1$, and due to the cross $\bf X$ to the right,  the local rules force that, in any stage, we do not have $\nu_{r+1}=\nu_r$ in the SSAF. When we arrive to the thick row $r-1$, one has $\nu_r<\nu_{r+1}$.
The contribution of the remain crosses in  $\mathcal{G}$, below,  will not change the inequality $\nu_r<\nu_{r+1}$, due to $\bf X$.

$(d)$ It follows from $(a)$ and $(b)$ by considering the  biword   $\overline w$   represented in $\mu:=\overline \lambda$.
\end{proof}
\subsection{ The bijection}
Given $\nu\in \mathbb{N}^n$, let  be the SSAF where, for all $i$, the column $i$ has   $\nu_i$ $i$'s. There exists, then, always a SSAF-pair for a given pair of shapes  in the same $\mathfrak{S}_n$-orbit.
 Let $w$ be a biword in lexicographic order on the alphabet $[n]$, and $\Phi(w)=(F,G)$. The pair $(sh(F),sh(G))$ will be called the key-pair of $w$.
  One shows that the biwords whose biletters constitute a multiset of cells in a staircase possibly plus a layer of boxes, as  in  figure $(\star)$,
 are characterized by inequalities in the Bruhat order satisfied by its key-pair. Notice that in figure $(\star)$, $1\le r_{k}<\cdots< r_{p+1}<r_1<\cdots<r_p<n$. We start with the case $k=p$.
\begin{thm} \label{bijection} {\em (NW or SE layer)}
Let $w$ be a biword in lexicographic order on the alphabet $[n]$, with key-pair $(\nu,\beta)$.
Let  $0\le k<n$, and $1\le r_1<\cdots<r_k<n$. The following conditions are equivalent

$(a)$ $w$ consists of a multiset of cells in the staircase of size $n$ plus
the $k$ cells
 ${n-r_1+1\choose r_1+1}$$,\ldots,$

 ${n-r_k+1\choose r_k+1}$, each with multiplicity at least one, as in figure $(\star)$ with $k=p$.

$(b)$  $\beta\leq  \omega s_{r_k}\cdots
s_{r_2}s_{r_1}\nu$, and $\beta\nleq  \omega s_{r_k}\cdots
\widehat s_{r_{i}}\cdots s_{r_1}\nu$, for $1\le i\le k$, where  $\,\,\widehat\empty\,\,$ means omission.

$(c)$
$s_{n-r_1}\cdots s_{n-r_k}\beta\leq  \omega\nu$, and $s_{n-r_1}\cdots \widehat s_{n-r_{i}}\cdots s_{n-r_k}\beta\nleq
\omega \nu$,{ for } $1\le i\le k$.
\end{thm}
\begin{proof} By induction on $k$. For $k=0$, it is the staircase in \cite{lascouxcrystal}, and in \cite{ejc}. For $k>0$, we use Theorem \ref{corcor1}, (c), to prove $(a) \Leftrightarrow(b)$. The detailed proof can be found in \cite{cim}. Once we have proved this,  $(a) \Leftrightarrow(c)$ follows now from Theorem \ref{corcor1}, (d).\end{proof}
As a consequence of this theorem, if  $\nu$, $\beta\in \mathbb{N}^n$ satisfy the inequalities in $(b)$, then
$\beta\nleq \omega s_{i_t}\cdots s_{i_1}\nu$, for $0\le t<k$, and $ i_1<\dots<i_t$ any subsequence of $ r_1<\cdots<r_k$.
Hence,  Lemma \ref{1layer-easy} 
 implies.
\begin{lem}\label{equi-NW-SE-mix}  
 Let  $0\le p\le k<n$, and $1\le r_1<\cdots<r_k<n$.
The three conditions are equivalent

$(a)$ $\beta\nleq  \omega s_{r_k}\cdots \widehat s_{{i}}\cdots s_{r_1}\nu, \text{ for } i=1,\ldots,k$, and $\beta\leq  \omega s_{r_k}\cdots s_{r_1}\nu.$

$(b)$ $ s_{n-r_1}\cdots \widehat s_{n-r_{i}}\cdots s_{n-r_k}\beta\nleq
\omega \nu, \text{ for } i=1,\ldots,k$, and $s_{n-r_1}\cdots s_{n-r_k}\beta\leq  \omega\nu.$

$(c)$ $ s_{n-r_{k-p+1}}\cdots s_{n-r_{k}}\beta\nleq  \omega s_{r_{k-p}}\cdots
\widehat s_{r_{i}}\cdots s_{r_1}\nu$, { for } $i=1,\ldots, k-p,$
$ s_{n-r_{k-p+1}}\cdots $ $\widehat s_{n-r_{i}}\cdots $ $s_{n-r_{k}}$ $\beta\nleq
\omega s_{r_{k-p}}\cdots s_{r_1}\nu$, { for } $i=k-p+1,\ldots,k,$ and
$ s_{n-r_{k-p+1}}\cdots s_{n-r_{k}}\beta\leq  \omega s_{r_{k-p}}\cdots
s_{r_1}\nu$.
\end{lem}
Theorem \ref{bijection} can now be written according to   figure $(\star)$, with the generic cutting line.  The condition $r_{p+1}-r_{1}>1$, if $p>0$, allows to use the commutation relation of $\mathfrak{S}_n$.
 \begin{thm} \label{bijection-NW-SE} {\em (NW-SE  layer)}
Let $w$ be a biword in lexicographic order on the alphabet $[n]$, with key-pair $(\nu,\beta)$.
Let  $0\le p\le k<n$,  and $1\le r_{k}<\cdots< r_{p+1}<r_1<\cdots<r_p<n$, where $r_1-r_{p+1}>1$, if $p>0$. The following conditions are equivalent

$(a)$ $w$ consists of a multiset of cells in the staircase of size $n$ and
the $k$ cells
 ${e_k+1\choose r_k+1}$$,\ldots,$${e_{p+1}+1\choose r_{p+1}+1}$, ${n-r_{1}+1\choose r_{1}+1}$$,\ldots,$${n-r_{p}+1\choose r_{p}+1}$ above it, each with multiplicity at least one, as shown in figure $(\star)$.

$(b)$ $s_{e_{k}}\cdots s_{e_{p+1}}\beta\nleq  \omega s_{r_p}\cdots
\widehat s_{r_{i}}\cdots s_{r_1}\nu, \text{ for } i=1,\ldots, p,$
$ s_{e_{k}}\cdots \widehat s_{e_{i}}\cdots s_{e_{p+1}}\beta\nleq
\omega s_{r_p}\cdots s_{r_1}\nu$, { for } $ i=p+1,\ldots,k,$
and $ s_{e_k}\cdots s_{e_{p+1}}\beta\leq  \omega s_{r_p}\cdots
s_{r_1}\nu.$
\end{thm}
\section{ A  combinatorial Cauchy kernel expansion
over near staircases}
\label{sec:last}
 Let $\lambda$ be the Ferrers shape as in figure $(\star)$. We  are finally equipped for the bijective proof
 of the $F_\lambda$ expansion. Before, we still need some definitions and a technical  lemma.
 Let SSYT$_n$ and SSAF$_n$ denote the set of all SSYTs and SSAFs on the alphabet $[n]$, respectively.
We assume the reader familiar with the definitions and basic properties of Demazure operators, or divided differences, $\pi_i$ and $\hat\pi_i$,  key polynomials, or Demazure characters, $\kappa_\nu$, and Demazure atoms  $\widehat \kappa_\nu$, where $\nu\in\mathbb{N}^n$, and refer the reader to \cite{lascouxdraft} and \cite{reiner}. In particular, recall  the action of Demazure operators $\pi_i$  on the key polynomial $\kappa_\nu$ and
 on  the Demazure atom  $\widehat\kappa_\nu$:
$\pi_i\kappa_\nu= \kappa_{s_i\nu}$, $\mbox{if } \nu_i>\nu_{i+1}$, and else, $ \kappa_\nu$;
and $\pi_{i} \widehat{\kappa}_\nu=
\widehat{\kappa}_{s_{i}\nu}+\widehat{\kappa}_{\nu}$, if $\nu_{i}>\nu_{i+1}$, else, $ \widehat{\kappa}_{\nu}$, if $\nu_{i}=\nu_{i+1}$
and, $0$, if $\nu_{i}<\nu_{i+1}$.
Recall also the combinatorial formulas in terms of SSYTs, \cite{lascouxkeys}, and SSAFs, \cite{masondemazure}, 

\vskip0.3em

$\displaystyle{\hat\kappa_\nu=
    \sum_{\begin{smallmatrix}T\in SSYT_n\\K_+(T)=key(\nu)\end{smallmatrix} }x^T=\sum_{\begin{smallmatrix}F\in SSAF_n\\sh(F)=\nu\end{smallmatrix}}x^F, }$ and $\displaystyle{\kappa_\nu=\sum_{\begin{smallmatrix}T\in SSYT_n\\K_+(T)\le key(\nu)\end{smallmatrix}}x^T=\sum_{\begin{smallmatrix}F\in SSAF_n\\sh(F)\le \nu\end{smallmatrix}}x^F.}$

\vskip0.3em
Given $0\le p\le k< n$, let  $k-p< r_1<\dots <r_{p}<n$ and $p<e_{p+1}<\dots <e_k<n$. For each $(z,t)\in [0, p]\times [0, k-p]$, and each $H_z=\{i_1<\dots<i_z\}\in {[p]\choose z}$, and each $M_t=\{j_1<\dots<j_t\}\in {[p+1 ~ k]\choose t}$,
 define, for $z=t=0$,
$\displaystyle 
\mathcal{A}^{\emptyset,\emptyset}=\left\{\begin{smallmatrix}(F,G)\in SSAF_n^2:\end{smallmatrix} \begin{smallmatrix}
sh(G)\leq  \omega sh(F)\end{smallmatrix}\right\}$, and else

\vskip0.3em

$\displaystyle
\mathcal{A}_{z,t}^{H_z,M_t}=\left\{\begin{smallmatrix}(F,G)\in SSAF_n^2:\end{smallmatrix} \begin{smallmatrix}
 s_{e_{j_t}}\cdots s_{e_{j_{1}}}sh(G)\nleq  \omega s_{r_{i_z}}\cdots \widehat s_{r_{i_m}}\cdots s_{r_{i_{1}}}sh(F),~ m=1,2, \ldots,z\\ \\
s_{e_{j_t}}\cdots \widehat s_{e_{j_l}}\cdots s_{e_{j_{1}}}sh(G)\nleq  \omega s_{r_{i_z}}\cdots s_{r_{i_1}}sh(F),~ l=1,2, \ldots,t\\ \\
 s_{e_{j_t}}\cdots s_{e_{j_{1}}}sh(G)\leq  \omega s_{r_{i_z}}\cdots s_{r_{i_1}}sh(F)\end{smallmatrix}\right\}.
$

\vskip0.3em
\vskip0.2em

\begin{lem}\label{2-separationgeneral} {\em(NW-SE Separation)}
Given $0\le p\le k< n$ and  $k-p< r_1<r_2<\dots <r_{p}<n$ and $p<e_{p+1}<\dots <e_k<n$. For each $(z,t)\in [0, p]\times [0, k-p]$, and each $H_z=\{i_1<i_2<\dots<i_z\}\in {[p]\choose z}$, and each $M_t=\{p+2\le j_1<j_2<\dots<j_t\}\in {[p+2 ~ k]\choose t}$,
 let $M^1_{t+1}:=\{p+1\}\cup M_t$, and
\begin{eqnarray}\nonumber
\mathcal{B}_{z,t}^{H_z,M_t}:=\left\{\begin{smallmatrix}(F,G)\in SSAF_n^2:\end{smallmatrix} \begin{smallmatrix}
sh(G)_{e_{p+1}}<sh(G)_{e_{p+1}+1}\\ \\
 s_{e_{j_t}}\cdots s_{e_{j_{1}}}s_{e_{p+1}}sh(G)\nleq  \omega s_{r_{i_z}}\cdots \widehat s_{r_{i_m}}\cdots s_{r_{i_{1}}}sh(F),~ m=1,2, \ldots,z\\ \\
s_{e_{j_t}}\cdots \widehat s_{e_{j_l}}\cdots s_{e_{j_{1}}}s_{e_{p+1}}sh(G)\nleq  \omega s_{r_{i_z}}\cdots s_{r_{i_1}}sh(F),~ l=1,2, \ldots,t\\ \\
 s_{e_{j_t}}\cdots s_{e_{j_{1}}}s_{e_{p+1}}sh(G)\leq  \omega s_{r_{i_z}}\cdots s_{r_{i_1}}sh(F)\end{smallmatrix}\right\}.
\end{eqnarray}
Then
$\displaystyle\mathcal{B}_{z,t}^{H_z,M_t}=\{(F,G)\in \mathcal{A}_{z,t}^{H_z,M_t}:sh(G)_{e_{p+1}}< sh(G)_{e_{p+1}+1}\}\cup \mathcal{A}_{z,t+1}^{H_z,M^1_{t+1}}.$
\end{lem}

Let $x=(x_1,\ldots,x_n)$ and $y=(y_1,\ldots, y_n)$ be two sequences of indeterminates. If  $\pi_{e_{p+1}}$ is the Demazure operator with respect to $y$,   one has

$
\displaystyle
\label{inductt+++}\sum_{\beta\in\mathbb{N}^n}\pi_{e_{p+1}}\sum_{\begin{smallmatrix}(F,G)\in \mathcal{A}_{z,t}^{H_z,M_t}\\sh(G)=\beta\end{smallmatrix}}x^Fy^G=
\sum_{\begin{smallmatrix}(F,G)\in \mathcal{A}^{H_z,M_t}_{z,t}\\ sh(G)_{e_{p+1}}\ge sh(G)_{e_{p+1}+1}\end{smallmatrix}}x^Fy^G+
\sum_{\begin{smallmatrix}(F,G)\in \mathcal{B}_{z,t}^{H_z,M_t}\end{smallmatrix}}x^Fy^G.\quad (2)
$
\begin{thm}
\label{right+inside+inside} 
 Let $0\le p\le k<n$. Let $k-p< r_1<r_2<\dots <r_{p}<n$ and $p<e_{p+1}<\dots <e_k<n$,  with   $p+1<e_{p+1}$, if $p>0$.  For each $(z,t)\in [0, p]\times [0, k-p]$, let  $(H_z,M_t)\in {[p]\choose z}\times {[p+1,k]\choose t}$.
  Then

$\displaystyle 
F_\lambda=\sum_{(F,G)\in \mathcal{A}^{\emptyset,\emptyset}}x^Fy^G
+ \displaystyle\sum_{z=1}^p \sum_{H_z\in{[p]\choose z}}\sum_{\begin{smallmatrix}(F,G)\in \mathcal{A}_{z,0}^{H_z,\emptyset}\end{smallmatrix}}x^Fy^G
+\sum_{t=1}^{k-p}\sum_{M_t \in{[p+1,k]\choose t}}\sum_{\begin{smallmatrix}(F,G)\in \mathcal{A}_{0,t}^{\emptyset, M_t}\end{smallmatrix}}x^Fy^G$

$+\displaystyle{\sum_{\begin{smallmatrix}(z,t)\in [p]\times [k-p]\end{smallmatrix}}\sum_{(H_z,M_t)}\sum_{\begin{smallmatrix}(F,G)\in \mathcal{A}_{z,t}^{H_z,M_t}\end{smallmatrix}}x^Fy^G
=\sum_{\nu\in \mathbb{N}^n}\pi_{r_1}\dots\pi_{r_p}\widehat\kappa_\nu(x)\pi_{e_{p+1}}\dots\pi_{e_k}\kappa_{\omega\nu}(y)}$

$
=\displaystyle
{\sum_{\nu\in \mathbb{N}^n}\widehat\kappa_\nu(x)\pi_{n-r_p}\cdots\pi_{n-r_1}\pi_{e_{p+1}}\cdots\pi_{e_k}\kappa_{\omega\nu}(y)=\sum_{\nu\in \mathbb{N}^n}\pi_{n-e_{k}}\cdots\pi_{n-e_{p+1}}\pi_{r_1}\cdots\pi_{r_p}\widehat\kappa_\nu(x)\kappa_{\omega\nu}(y).}$
\end{thm}
\begin{proof}   The proof is by double induction on  $p\ge 0$ and $k-p\ge 0$. For $p\ge 0$ and $k-p$$= 0$, and {\em vice-versa}, see
\cite{cim}.
Let $p,\,k-p\ge 1$.
One has,

$
\displaystyle\sum_{\nu\in \mathbb{N}^n}\pi_{r_1}\dots\pi_{r_p}\widehat\kappa_\nu(x)\pi_{e_{p+1}}\dots\pi_{e_{k}}\kappa_{\omega\nu}(y)=
\pi_{e_{p+1}}\left(\displaystyle\sum_{\nu\in \mathbb{N}^n}\pi_{r_1}\dots\pi_{r_p}\widehat\kappa_\nu(x)\pi_{e_{p+2}}\dots\pi_{e_{k}}\kappa_{\omega\nu}(y)\right)
$, by induction,

 $
=\pi_{e_{p+1}}\left(
\displaystyle\sum_{\begin{smallmatrix}0\le z\le p\\0\le t\le k-p-1\end{smallmatrix}}\sum_{(H_z,M_t)}\sum_{\begin{smallmatrix}(F,G)\in \mathcal{A}_{z,t}^{H_z,M_t}\end{smallmatrix}}x^Fy^G\right),\quad\text{$(H_z,M_t)\in  {[p]\choose z}\times{[p+2,k]\choose t}$},
$

$=
\displaystyle
\sum_{\begin{smallmatrix}0\le z\le p\\0\le t\le k-p-1\end{smallmatrix}}\sum_{(H_z,M_t)}\left(\sum_{\beta\in \mathbb{N}^n}\pi_{e_{p+1}}\sum_{\begin{smallmatrix}(F,G)\in \mathcal{A}_{z,t}^{H_z,M_t}\\ sh(G)=\beta\end{smallmatrix}}x^Fy^G\right),\quad\text{using $(2)$},
$

 $\label{uselemma22} =\displaystyle\sum_{\begin{smallmatrix}0\le z\le p\\0\le t\le k-p-1\end{smallmatrix}}\sum_{(H_z,M_t)}\left(
\sum_{\begin{smallmatrix}(F,G)\in \mathcal{A}^{H_z,M_t}_{z,t}\\ sh(G)_{e_{p+1}}\ge sh(G)_{e_{p+1}+1}\end{smallmatrix}}x^Fy^G+
\sum_{\begin{smallmatrix}(F,G)\in \mathcal{B}_{z,t}^{H_z,M_t}\end{smallmatrix}}x^Fy^G\right),\;\text{ from Lemma \ref{2-separationgeneral}},
$

$
=
\displaystyle\sum_{\begin{smallmatrix}0\le z\le p\\0\le t\le k-p-1\end{smallmatrix}}\sum_{(H_z,M_t)}
\left(\sum_{\begin{smallmatrix}(F,G)\in \mathcal{A}_{z,t}^{H_z,M_t}\\ sh(G)_{e_{p+1}}\ge sh(G)_{e_{p+1}+1}\end{smallmatrix}}x^Fy^G+
{}
\displaystyle
\sum_{\begin{smallmatrix}(F,G)\in \mathcal{A}_{z,t}^{H_z,M_t}\\ sh(G)_{e_{p+1}}<sh(G)_{e_{p+1}+1}\end{smallmatrix}}x^Fy^G+
\sum_{\begin{smallmatrix}(F,G)\in \mathcal{A}_{z,t+1}^{H_z,M^1_{t+1}}\end{smallmatrix}}x^Fy^G\right)\\
=\displaystyle\sum_{\begin{smallmatrix}0\le z\le p\\0\le t\le k-p-1\end{smallmatrix}}\sum_{(H_z,M_t)}
\left(\sum_{\begin{smallmatrix}(F,G)\in \mathcal{A}_{z,t}^{H_z,M_t}\end{smallmatrix}}x^Fy^G+
\sum_{\begin{smallmatrix}(F,G)\in \mathcal{A}_{z,t+1}^{H_z,M^1_{t+1}}\end{smallmatrix}}x^Fy^G\right)
$

$
=\displaystyle\sum_{\begin{smallmatrix}0\le z\le p\\0\le t\le k-p-1\end{smallmatrix}}\sum_{(H_z,M_t)}
\left(\sum_{\begin{smallmatrix}(F,G)\in \mathcal{A}_{z,t}^{H_z,M_t}\end{smallmatrix}}x^Fy^G\right)
$
$
+
\displaystyle\sum_{\begin{smallmatrix}0\le z\le p\\0\le t\le k-p-1\end{smallmatrix}}\sum_{(H_z,M_t)}\left(
\sum_{\begin{smallmatrix}(F,G)\in \mathcal{A}_{z,t+1}^{H_z,M^1_{t+1}}\end{smallmatrix}}x^Fy^G\right)
$

$
={\displaystyle{
\sum_{(F,G)\in \mathcal{A^{\emptyset,\emptyset}}}x^Fy^G+\sum_{\begin{smallmatrix}(z,t)\in [0,p]\times [0,k-p]\\ (z,t)\neq (0,0)\end{smallmatrix}}\sum_{(H_z,M_t)}\sum_{\begin{smallmatrix}(F,G)\in \mathcal{A}_{z,t}^{H_z,M_t}\end{smallmatrix}}x^Fy^G}},\;\text{with $(H_z,M_t)\in  {[p]\choose z}\times{[p+1,k]\choose t}$}.
$

 Identifying $x_iy_j$ with the biletter ${j\choose i}$, one has three types of biwords:  inside the staircase, consisting only of the extra biletters in the NW part, or  the extra biletters in the SE part. 
  Using bijection \ref{bijection-NW-SE}, their concatenation gives,

\hspace{-0.7cm}$$F_\lambda=F_\rho\prod_{i=1}^p(1-x_{{r_i}+1}y_{n-r_i+1})^{-1}\prod_{j=p+1}^k(1-x_{{r_j}+1}y_{e_j+1})^{-1}
=\displaystyle\sum_{(F,G)\in \mathcal{A^{\emptyset,\emptyset}}}x^Fy^G+\displaystyle\sum_{\begin{smallmatrix}(z,t)\in [0,p]\times [0,k-p]\\ (z,t)\neq (0,0)\end{smallmatrix}}\sum_{(H_z,M_t)}\sum_{\begin{smallmatrix}(F,G)\in \mathcal{A}_{z,t}^{H_z,M_t}\end{smallmatrix}}x^Fy^G.$$
 \end{proof}
 
 \bibliographystyle{abbrvnat}
\bibliography{sample}
\label{sec:biblio}
\vspace{0.5cm}
oazenhas@mat.uc.pt (Olga Azenhas), CMUC, Department of Mathematics, University of Coimbra, 3001-501 Coimbra, Portugal.\\

 \hspace{-0.6cm}aramee@mat.uc.pt (Aram Emami) CMUC, Department of Mathematics, University of Coimbra, 3001-501 Coimbra, Portugal,  and Department of Mathematics, University of Fasa, Iran.

 \end{document}